\documentclass[12pt,reqno]{article}
\usepackage{fullpage}
\usepackage{mathtools}
\usepackage{enumitem}
\usepackage{hyperref}
\usepackage{amsthm,amsfonts,amssymb,euscript,mathrsfs}

\usepackage{color}
\usepackage{xcolor}

\newcommand{\R}{\mathbb{R}}

\newcommand{\Z}{\mathbb{Z}}
\newcommand{\N}{\mathbb{N}}

\renewcommand{\d}{\mathrm{d}}
\usepackage{stmaryrd}
\date{}
\theoremstyle{plain}

\newtheorem{lem}{Lemma}[section]
\newtheorem{thm}{Theorem}[section]
\newtheorem{prop}{Proposition}[section]
\newtheorem{coro}{Corollary}[section]
\newtheorem{mydef}{Definition}[section]
\newtheorem{remark}{Remark}[section]
\numberwithin{equation}{section}

\newcommand{\ffi}{\varphi}

\newcommand{\e}{\varepsilon}
\newcommand{\dr}{\partial}
\newcommand{\tdr}{\overline{\partial}}

\newcommand{\Ll}{\mathscr{L}}

\newcommand{\T}{\mathbb{T}}

\newcommand{\GO}[1]{O\left( #1 \right)}

\renewcommand{\l}{\left\|}
\renewcommand{\r}{\right\|}

\newcommand{\half}{\frac{1}{2}}
\newcommand{\enstq}[2]{\left\{#1~\middle|~#2\right\}} 

\title{Global existence of high-frequency solutions to a semi-linear wave equation with a null structure}
\author{Arthur Touati}

\begin{document}

\maketitle

\begin{abstract}
We study the propagation of a compactly supported high-frequency wave through a semi-linear wave equation with a null structure. We prove that the self-interaction of the wave creates harmonics which remain close to the light-cone in the evolution. By defining a well-chosen ansatz, we describe precisely those harmonics. Moreover, by applying the vector field method to the equation for the remainder in the ansatz, we prove that the solution exists globally. The interaction between the dispersive decay of waves and their high-frequency behaviour is the main difficulty, and the latter is not compensated by smallness of the initial data, allowing us to consider the high-frequency limit where the wavelength tends to 0.
\end{abstract}

\tableofcontents

\section{Introduction}

\subsection{Presentation of the result}

In this article we are interested in proving the global existence of \textit{high-frequency solutions} to the following semi-linear wave system
\begin{align}
\Box \Phi = Q(\dr \Phi ,\dr \Phi) \label{main equation}
\end{align}
on $\R_+\times \R^3$ where $\Box= -\dr_t^2 + \Delta$ is the standard d'Alembertian operator and $Q$ has the null structure (see Definition \ref{def null form}).  The initial data we consider are highly oscillatory in the radial direction. This corresponds to a large data regime, which is to be compared with the well-known small data regime, see the discussion below. Let us give a rough version of the theorem proved in this article (see Theorem \ref{premier thm} for a precise version):

\begin{thm}\label{rough}
Consider sufficiently regular functions $(F_0,\ffi_0,\ffi_1)$ of size $\e$ on $\R^3$ and let $\lambda>0$. If $\e$ is small enough (independent of $\lambda$), there exists a global solution $\Phi_\lambda$ to $\Box u= Q(\dr u,\dr u)$ of the form:
\begin{equation}
    \Phi_{\lambda} = \ffi+\lambda F\cos\left(\frac{t-r}{\lambda}\right) +\lambda^2F_{\lambda}\label{ansatz vague}
\end{equation}
where
\begin{itemize}
\item $\ffi$ satisfies $\Box\ffi=Q(\dr\ffi,\dr\ffi)$ and $(\ffi,\dr_t\ffi)_{|t=0}=(\ffi_0,\ffi_1)$,
\item $F$ satisfies $\left(\dr_t+\dr_r+\frac{1}{r}\right)F=F\dr\ffi$ and $F_{|t=0}=F_0$,
\item $F_\lambda$ satisfies $|F_\lambda|\lesssim\e$.
\end{itemize}
\end{thm}

In the next section we will review in detail our strategy of proof,  for now let us just mention that the global existence of solutions to \eqref{main equation} on $\R^{3+1}$ with any (regular and small enough) initial data is a classical result of Klainerman (see \cite{MR837683}), which introduced the concept of null quadratic forms precisely for this purpose, and Christodoulou (see \cite{Christodoulou1986}) which used conformal compactification. Our proof relies crucially on the vector field method Klainerman introduced, but our theorem is \textit{not} a consequence of his result for the following reasons:
\begin{itemize}
\item Theorem \ref{rough} can be summarized as follows: take some oscillating initial data containing terms like $\lambda \cos\left(\frac{r}{\lambda}\right)$, then the unique global solution arrising from those initial data actually presents the same oscillating behaviour. Unlike Klainerman, we not only prove global existence but propagate the shape of the high-frequency ansatz.
\item Another crucial aspects of Theorem \ref{rough} is that the smallness threshold is independent of $\lambda$. One interesting aspects of high-frequency solutions of the form \eqref{ansatz vague} is that their $H^s$ norm is large for any $s>1$, if $\lambda$ is close to 0. Indeed, by differentiating \eqref{ansatz vague} strictly more than once we obtain negative powers of $\lambda$. Therefore, if one wants to apply Klainerman's result with a fixed but small $\lambda$ one needs to counterbalance the size of $\frac{1}{\lambda}$ with the smallness of the initial data (concretely by assuming something like $\e \ll\lambda$), see Remark \ref{limit}. Therefore, our work gives an example of how one can relax the smallness assumption of Klainerman's result.
\end{itemize}

\par\leavevmode\par
Our motivations for studying solutions of the form \eqref{ansatz vague} to non-linear wave equations like \eqref{main equation} come from Burnett's conjecture in general relativity and are presented at the end of this introduction (see Section \ref{section GR}). However, highly oscillating solution to hyperbolic equations have been intensely studied in the field of \textit{geometric optics}. In 1957, Lax laid the foundations of this field in his article \cite{MR97628}, where he proved that linear hyperbolic equations admit WKB solutions.  Named after the physicists Wentzel, Kramers and Brillouin, WKB solutions were first introduced to understand the behaviour of quantum systems in a semi-classical regime, the main idea being to write some part of the solution as an asymptotic expansion in terms of the small parameter $\hbar$. In geometric optics, the small parameter is a wavelength.

In this article, we prove global existence of a highly oscillating solutions of a \textit{non-linear} wave equation.  The construction of high-frequency ansatz for non-linear systems has been first conducted by Choquet-Bruhat in \cite{MR255964}, where she applies her result to deduce the ill-posedness of the Cauchy problem in $C^1$ for the systems she considers. In \cite{MR1073774}, Joly and Rauch extend the approximation procedure to higher space dimension. A complete overview of geometric optics can be found in the book of Rauch \cite{MR2918544}.  

In the limit $\lambda\to 0$, solutions of the form \eqref{ansatz vague} explode in any $W^{k,p}$ if $k>0$. Therefore, Theorem \ref{rough} allows us to consider large initial data for which we are still able to prove global existence. In \cite{MR3463418},  Wang and Yu also consider the large data regime by studying the \textit{short pulse} ansatz,  inspired by Christodoulou and his study of black holes formation.

\subsection{Strategy of proof}\label{subsection long time}

Let us present our strategy of proof, which can be viewed as a high-frequency adaptation of Klainerman's vector field method.  

\subsubsection{Global existence for non-linear wave equations}\label{section 121}

In spatial dimension strictly larger than 3, the standard decay of waves is enough to ensure the global existence under the small initial data assumption of solutions to wave equations of the form
\begin{align}
\Box u=(\dr u)^2.\label{wave}
\end{align}
However, since the counter example of John (see \cite{MR600571}), it is known that solutions to \eqref{wave} on $\R^{3+1}$ don't necessarily exist for all time, even in the small initial data regime.  The globality of solutions depends on the actual structure of the quadratic non-linearity. In 1984, Klainerman introduced the famous \textit{null condition} (see \cite{MR837683} and Definition \ref{def null form}).  The null condition can be summarized as follows: in every product $\dr u\dr u$ at least one of the derivatives is \textit{tangent to the light cone}, and therefore enjoys better decay. This restriction on the non-linearity in \eqref{wave} is sufficient (though not necessary) to obtain the existence of global solutions for small initial data.

\par\leavevmode\par
In order to prove global existence of a solution, one powerful method is the vector field method, which goes back to Klainerman (see \cite{MR784477}). Instead of using standard derivatives $\dr_\alpha$ to derive energy estimates, the heart of the method consists in using special vector fields linked to the symmetries of the Minkowski spacetime (see \eqref{minkowski vector fields}).  These vector fields are interesting for two reasons: they enjoy nice commutation properties with the wave operator and a more refined version of the classic Sobolev embedding $H^s\subset L^\infty$ for $s>\frac{3}{2}$ holds when replacing the standard derivatives in the $H^s$ norm by the Minkowski vector fields (the so-called Klainerman-Sobolev inequality, see Proposition \ref{prop WKNS}). This refined Sobolev embedding comes with decay rates in the null coordinates $t+r$ and $t-r$, which is the key argument to prove global existence.

\par\leavevmode\par
This method can be improved by using a \textit{weighted} energy estimate with an increasing function of $q=t-r$ as a weight. This weight implies the presence of an additional term on the LHS of the usual energy estimates (see Lemma \ref{weighted energy estimate}), allowing to better control some of the derivatives of the solution (in our case the derivatives tangent to the light cone). This is known as the ghost weight method and was first introduced by Alinhac in \cite{MR1856402}.

\subsubsection{The high-frequency hierarchy}

The main particularity of our result is the use of an high-frequency ansatz with a precise description of the oscillating terms, displaying what we call a \textit{half-chessboard shape}. Let $\Phi_\lambda$ be the solution of \eqref{main equation} admitting the following formal high-frequency expansion
\begin{align}
\Phi_\lambda(t,x) \sim \sum_{k\geq 0}\lambda^k\Phi^{(k)}\left(t,x,\frac{t-r}{\lambda}\right).\label{formal}
\end{align}
By half-chessboard shape, we mean that each $\Phi^{(i)}(t,x,\theta)$ contains several oscillating terms in the $\theta$ variable according to the following pattern:
\begin{align}
\Phi^{(0)}& = 1\nonumber
\\ \Phi^{(1)}& = \cos(\theta)\nonumber
\\ \Phi^{(2)}& = \sin(\theta) + \cos(2\theta)\label{pattern}
\\ \Phi^{(3)}& = \cos(\theta) + \sin(2\theta) + \cos(3\theta)\nonumber
\\&\quad \vdots\nonumber
\end{align}
\begin{remark}
In \eqref{pattern},  we only wrote down the oscillating terms, i.e the dependence of $\Phi^{(i)}$ on its third variable, and by $\Phi^{(0)} = 1$ we mean that it does not actually depends on the third variable, i.e it is not oscillating.
\end{remark}

The main challenge in defining a pattern of oscillation for each order is to capture the creation of harmonics due to the non-linear part of \eqref{main equation}. It is not trivial that the half-chessboard shape is "stable" under quadratic non-linearity of the form $(\dr\Phi_\lambda)^2$, i.e that those terms only create harmonics already present in the ansatz for $\Phi_\lambda$. This will be rigourously proved in Section \ref{section Construction of the ansatz}, where we present the exact form of the ansatz for $\Phi_\lambda$ (see \eqref{ansatz}).  The non-linear interactions also create non-oscillating terms which can't be absorbed by the ansatz, unless we add to \eqref{pattern} some non-oscillating fields. To lighten this introduction we choose not to show them yet.

\par\leavevmode\par

From a differential point of view, plugging oscillating terms $\lambda^k\cos\left(\frac{u}{\lambda}\right)F$ (with $u$ a phase to be chosen later) in the wave operator gives
\begin{align}
\Box\left(\lambda^k \cos\left(\frac{u}{\lambda}\right)F  \right) & = -\lambda^{k-2}\left(m^{\alpha\beta}\dr_\alpha u \dr_\beta u \right)\cos\left(\frac{u}{\lambda}\right)F  \label{intro exemple}
\\&\quad - \lambda^{k-1}\sin\left(\frac{u}{\lambda}\right)\left(2m^{\alpha\beta}\dr_\alpha u \dr_\beta + \Box u \right)F\nonumber
\\&\quad + \lambda^k\cos\left(\frac{u}{\lambda}\right)\Box F\nonumber
\end{align}
where $m^{\alpha\beta}\dr_\alpha u \dr_\beta u=-(\dr_t u)^2 + |\nabla u|^2$. We recover what is called the \textit{geometric optics approximation}: in order to have $\Box\left(\lambda^k \cos\left(\frac{u}{\lambda}\right)F  \right)$ at the same order as $\lambda^k \cos\left(\frac{u}{\lambda}\right)F$ we need 
\begin{itemize}
\item $m^{\alpha\beta}\dr_\alpha u \dr_\beta u=0$, i.e $u$ to be a solution of the eikonal equation, or said differently we need the spacetime gradient of $u$, denoted $\dr u$, to be a null vector field for the Minkowski metric.  In the rest of the article we will choose $u=t-r$.
\item $\left(2m^{\alpha\beta}\dr_\alpha u \dr_\beta + \Box u \right)F=0$, i.e $F$ to be transported along the rays of $u$. If $u=t-r$, this rewrites
\begin{align*}
\left( \dr_t + \dr_r + \frac{1}{r}\right) F =0.
\end{align*}
\end{itemize}
The fact that $u$ satisfies the eikonal equation has another consequence, related to our pattern of oscillation \eqref{pattern}. Indeed since $Q$ is a null form and $\dr u$ is a null vector the following term vanishes
\begin{align}
\left( \sin\left(\frac{u}{\lambda}\right) F^{(1)} \right)^2Q(\dr u, \dr u) \label{backreaction}
\end{align}
where $F^{(1)}$ is the coefficient of $\Phi^{(1)}$. The presence of this term would totally break the half-chessboard shape since we would need $\Phi^{(1)}$ (as well as $\Phi^{(k)}$ for $k\geq 2$ because of the triangular structure) to contain all frequencies (see Remark \ref{bad remark}).

\par\leavevmode\par

The conclusion of this discussion is that solving \eqref{main equation} is equivalent to solving a \textit{triangular} hierarchy of \textit{linear} transport equations along the rays of $u$ for each coefficients in the high-frequency expansion of $\Phi_\lambda$, as well as a semi-linear wave equation for the remainder $h$,  which is added to the formal expansion \eqref{formal} to make it an exact solution of \eqref{main equation}. Our ansatz is therefore of the following schematic form
\begin{align}
\Phi_\lambda(t,x) = \sum_{k= 0}^K\lambda^k\Phi^{(k)}\left(t,x,\frac{t-r}{\lambda}\right) + h_\lambda\label{formal 2}
\end{align}
where $h_\lambda=\GO{\lambda^K}$ solves an equation of the form
\begin{align}
\Box h_\lambda & = Q(\dr h_\lambda , \dr h_\lambda) + \lambda^K  \Phi^{(\cdot)}  \label{intro eq h}
\end{align}
where $Q$ is the null form and $\lambda^K  \Phi^{(\cdot)} $ denotes terms of orders $\lambda^K$ depending polynomially on the high-frequency waves $\Phi^{(k)}$ and their derivatives.  Note that we deliberately forgot a lot of terms in \eqref{intro eq h}, since from the point of view of the dependence in $\lambda$ the quadratic terms in $\dr h_\lambda$ are the most dangerous.

\par\leavevmode\par
In \eqref{formal 2}, the degree of precision of the expansion defining $\Phi_\lambda$ is given by the integer $K$.  Our work shows that we need $K\geq 4$ to close the bootstrap procedure for \eqref{intro eq h}. This is due to the method used to prove global existence to \eqref{intro eq h}, i.e the vector field method. Indeed, we want $h_\lambda$ to satisfy bootstrap assumptions of the form $\l \dr Z^I h_\lambda \r_{L^2}\lesssim \lambda^{K-|I|}$ (where $Z$ denotes any Minkowskian vector fields, see \eqref{minkowski vector fields}). This mimics a high-frequency behaviour while $h_\lambda$ is not oscillating. In order to improve these bootstrap assumptions, we will recover decay from the Klainerman-Sobolev inequality. This inequality allows one to bound $|f|$ by $\l Z^3 f \r_{L^2}$ with appropriate decay rates, which we omit in this introduction (see Proposition \ref{prop WKNS}). The bootstrap assumptions for $h_\lambda$ therefore imply that $|\dr h_\lambda| \lesssim \lambda^{K-4}$. This explains the lower bound on the degree of precision of the ansatz $K\geq 4$.

\par\leavevmode\par
As a final remark on our strategy of proof, let us explain how we get decay for the high-frequency waves $\Phi^{(k)}$, which appear as source terms in \eqref{intro eq h}. As explained above, the high-frequency waves solve a hierarchy of transport equation of the form
\begin{align}
\left( \dr_t + \dr_r + \frac{1}{r} \right) f = g\label{transport modèle}
\end{align}
with a RHS depending on the lower order high-frequency waves (thus respecting the triangular structure of the system).  This equation rewrites $(\dr_t + \dr_r) (rf) = rg$, from which we get $|f|\lesssim r^{-1}$ if the RHS satisfies $|g|\lesssim r^{-3}$. See Proposition \ref{equation de transport prop} for a precise statement. In the hierarchy of transport equations, the most dangerous term is $\Box\Phi^{(k)}$, i.e comes from the third line in \eqref{intro exemple}. Schematically if we only consider this term we have
\begin{align*}
\left( \dr_t + \dr_r + \frac{1}{r} \right) \Phi^{(k+1)} & = \Box \Phi^{(k)}.
\end{align*}
Therefore, in order to apply Proposition \ref{equation de transport prop} we need $\left|\Box \Phi^{(k)}\right|\lesssim r^{-3}$. We can't get this decay by commuting $\Box$ with the transport operator $\dr_t + \dr_r + \frac{1}{r}$ since it would only give $\left|\Box \Phi^{(k)}\right|\lesssim r^{-1}$. Instead we compute and estimate $\Box \Phi^{(k)}$ directly from $\dr_t + \dr_r + \frac{1}{r}$ in Lemma \ref{d'alembertien}. See Remark \ref{remarque 5.2} for a more detailled discussion of this procedure.

\subsection{A toy model for general relativity} \label{section GR}

Here we present our principal motivation for studying the global existence of high-frequency solutions to \eqref{main equation}.

\subsubsection{The Einstein vacuum equations in wave coordinates}

The Einstein equations are the main equations of the general relativity theory, they spell out the link between the curvature of the spacetime and the matter and energy it contains. On a Lorentzian manifold $(\mathcal{M},g)$ and in the vacuum case they can be written in an elegant and compact form:
\begin{equation}\label{EVE1}
\mathrm{Ric}(g)=0,
\end{equation}
where $\mathrm{Ric}(g)$ is the Ricci tensor of the metric $g$. In order to define a initial value problem for this equation, we need to specify a coordinate system on $\mathcal{M}$. Noticed by Einstein in 1916 and used by Choquet-Bruhat to show local existence of the Einstein system in the 1950's, the \textit{wave coordinates} define the most common coordinate system used in the study of the Einstein equations. In this coordinate system, the equation \eqref{EVE1} becomes
\begin{equation}\label{EVE2}
\Box_g g_{\alpha\beta} = P_{\alpha\beta}(\dr g,\dr g),
\end{equation}
where $\Box_g$ is the wave operator associated to $g$ and $P_{\alpha\beta}$ is a quadratic non-linearity. Thus, the Einstein vacuum equations are recast as a quasi-linear hyperbolic system for the metric coefficients, which allows one to define and study a Cauchy problem (see \cite{MR53338} for the first proof of local well-posedness using wave coordinates). 

One of the major question in general relativity is the question of the \textit{stability} of particular solutions, such as the Minkowski spacetime, which, thanks to \eqref{EVE2} reduces to a hyperbolic long-time existence problem. As it was proved by Choquet-Bruhat in \cite{MR1770106}, the non-linearity in \eqref{EVE2} does not satisfy the null condition presented in Section \ref{subsection long time}. Note that the Einstein vacuum equations do satisfy a null condition through a dynamical gauge choice, see the seminal work \cite{CK}. However, the Einstein vacuum equations in wave gauge satisfy what Lindblad and Rodnianski call a \textit{weak null condition}, allowing them to prove the non-linear stability of the Minkowski spacetime in this gauge in \cite{MR2680391}.

Therefore, studying the semi-linear equation $\Box u= Q(\dr u,\dr u)$ where $Q$ is a null form is a toy model for the Einstein vacuum equations: it is a semi-linear equation instead of a quasi-linear equation, with the quadratic part being a null form instead of merely satisfying the weak null condition.

\subsubsection{Burnett's conjecture}

The Einstein vacuum equations are highly non-linear, which a priori allows \textit{backreaction} to occur. This phenomenon can be described as follows: if a sequence of metrics $(g_\lambda)_\lambda$ converges in some sense to a background metric $g_0$, and if for all $\lambda$ we have $\mathrm{Ric}(g_\lambda)=0$, do we have $\mathrm{Ric}(g_0)=0$ ? Said differently, can we pass to the limit in the Ricci tensor? Identifying the potential limit matter models occuring because of backreaction is called the Burnett conjecture, and was introduced by Burnett in \cite{bur89}.

If the convergence $g_\lambda\to g_0$ happens in a very strong topology then the answer is obviously positive, but if the convergence is weak enough, we might obtain non-zero contributions from quadratic terms like $\dr (g_\lambda - g_0)\dr (g_\lambda - g_0)$.  In the framework chosen by Burnett, the convergence is the following: $g_\lambda$ converges strongly in $L^\infty$ and $\dr g_\lambda$ converges weakly in $L^2$.  High-frequency ansatz of the form \eqref{formal}-\eqref{pattern} precisely enjoy this type of convergence, and have been used to tackle Burnett's conjecture locally in time by Huneau and Luk in \cite{MR3881199}. This fully motivates the study of high-frequency solutions to \eqref{main equation}.

\begin{remark}
Even though the high-frequency solution $\Phi_\lambda$ converges in the Burnett sense to $\Phi^{(0)}$, no backreaction occurs in the case of the equation \eqref{main equation}, meaning that the background limit solution $\Phi^{(0)}$ also satisfy \eqref{main equation} (see also \eqref{equation sur ffi} in the reduced system). As Choquet-Bruhat noticed in \cite{MR1770106}, this is linked to the null condition satisfied by \eqref{main equation}.  Indeed, backreaction could only occur through the non-oscillating contributions of the squared sinus in \eqref{backreaction}, but as we saw \eqref{backreaction} vanishes because $Q$ is a null form. 

The absence of backreaction for \eqref{main equation} takes nothing away from the study of high-frequency solutions to this equation, as it is a first step towards proving the stability of high-frequency perturbation of some background metric $g_0$ close to Minkowski for the equation \eqref{EVE2}, which would require to incorporate the ideas of \cite{MR2680391} to the high-frequency setting presented in this article.
\end{remark}

\section{Preliminaries}

In this section, we describe our geometric and analytic setting, and introduce the key estimates used in this article.

\subsection{Notations and function spaces}

\subsubsection{Coordinates and derivatives}

We work in $\R_+\times \R^3$ with the usual global coordinates system $(t=x_0,x_1,x_2,x_3)$. We also introduce the null coordinates
\begin{equation*}
s=t+r \quad \text{and}\quad q=r-t.
\end{equation*}
For $t\geq 0$ we use the standard notation $\Sigma_t\vcentcolon= \{t\}\times \R^3$. We also define the following subset of $\R_+\times\R^3$ for some $R>0$:
\begin{align*}
     A^R&\vcentcolon=\enstq{(t,x)\in\R_+\times\R^3}{R^{-1}\leq q\leq R},\\
     B^R&\vcentcolon=\enstq{(t,x)\in\R_+\times\R^3}{q\leq R},
\end{align*}
and we set $A_0^R\vcentcolon=A^R\cap \Sigma_0$.  If $x=(x_1,x_2,x_3)$ and $r=|x|$, recall that $\dr_r=\sum_{i=1,2,3}\omega_i\dr_i$ where $\omega_i=\frac{x_i}{r}$.  If $f:\R_+\times\R^3\longrightarrow\R$ is a scalar function, then we define its spacetime gradient by the vector $\dr f=(\dr_t f, \dr_1 f,\dr_2 f,\dr_3 f)$.
The "good" derivatives are defined to be 
\begin{equation*}
\tdr_0 = \frac{1}{2}(\dr_t+\dr_r)\quad\text{and}\quad \tdr_i=\dr_i-\frac{x_i}{r}\dr_r.
\end{equation*}
As we will see, they enjoy better decay than the other derivatives.  In this article we consider $\R^d$-valued functions for some $d\geq 1$, i.e vectorial functions of the form $u=(u_1,\dots, u_d)$. For such a function we define
\begin{align*}
|\dr u| = \sum_{\substack{\alpha=0,\dots,3\\i=1,\dots,d}}  |\dr_\alpha u_i| \quad\text{and}\quad |\tdr u| = \sum_{\substack{\alpha=0,\dots,3\\i=1,\dots,d}}  |\tdr_\alpha u_i| .
\end{align*}

\subsubsection{Null forms}
Now we define the class of quadratic non-linearties we consider in this article, i.e the classical null forms introduced by Klainerman.

\begin{mydef}\label{def null form}
A quadratic form $Q$ is said to be a \textit{null quadratic form} (also called null form in the sequel) if it is a linear combination of $Q_0$ and $Q_{\alpha\beta}$ where
\begin{equation*}
Q_{\alpha\beta}(X,Y)=X_\alpha Y_\beta -  X_\beta Y_\alpha \quad\text{and}\quad Q_{0}(X,Y) = -X_0 Y_0+ \delta^{ij}X_iY_j
\end{equation*}
for two vector fields $X,Y$ on $\R_+\times \R^3$, and $\alpha,\beta=0,1,2,3$.
\end{mydef}
One can show that the null forms defined in this way are precisely the quadratic forms on the tangent bundle of Minkowski spacetime vanishing when contracted twice with the same null vector field.

\par\leavevmode\par
Since the spacetime gradient of a scalar function is a vector field on $\R_+\times \R^3$, this definition allows us to consider $Q(\dr u, \dr v)$ for $u,v$ two \textit{scalar} functions.  However in this article, we would like to consider \textit{vectorial} functions. Therefore if $u:\R_+\times \R^3\longrightarrow \R^d$ with $u=(u_1,\dots,u_d)$ and $v:\R_+\times \R^3\longrightarrow \R^{d'}$ with $v=(v_1,\dots,v_{d'})$,  and if $Q_{k\ell}$ are null forms, we define
\begin{align}
Q(\dr u,\dr v) \vcentcolon = \sum_{\substack{k=1,\dots,d\\\ell=1,\dots,d'}}Q_{k\ell}(\dr u_k,\dr v_\ell).\label{null form matrice}
\end{align}
With this notation, we can write down properly the system studied in this article. For $d\geq 1$, we want to solve the system
\begin{align*}
\Box u_i = Q_i (\dr u , \dr u)
\end{align*}
where $u:\R_+\times \R^3\longrightarrow \R^d$ is a vectorial function, i.e $u=(u_1,\dots,u_d)$, and where $Q_i$ is as in \eqref{null form matrice} for $i=1,\dots,d$. Though the Einstein equations motivates the study of a system of wave equations in this article, this plays no particular role here and from now on (with the exception of the statement of the main result) we will drop the indexes relative to $\R^d$ and simply use the condensed notation
\begin{align*}
\Box u = Q(\dr u,\dr u).
\end{align*}

\par\leavevmode\par
The null forms are in some sense the best quadratic non-linearity we can hope for in the sense that in each product $\dr u\dr u$, at least one of the derivatives is a good one, meaning that:
\begin{equation}\label{good null form}
|Q(\dr u, \dr v)|\lesssim |\dr u||\tdr v|+|\tdr u||\dr v|.
\end{equation}

\subsubsection{Function spaces}
In terms of function spaces, we use the usual $L^p$, $C^m$ and Sobolev spaces defined on $\R^3$, together with weighted Sobolev spaces, defined by 
\begin{mydef}\label{wss}
Let $m\in\N$, $1<p<\infty$, $\delta\in\R$. The weighted Sobolev space $W^{m,p}_{\delta}$ is the completion of $C^{\infty}_0$ under the norm
\begin{equation*}
    \| u\|_{W^{m,p}_{\delta}}=\sum_{|\beta|\leq m} \left\| \langle x \rangle^{\delta+\beta}\nabla^{\beta}u \right\|_{L^p}.
\end{equation*}
We will use the notation $H^m_{\delta}=W^{m,2}_{\delta}$.  The weighted Hölder space $C^m_{\delta}$ is the completion of $C^m$ under the norm
\begin{equation*}
    \| u\|_{C^m_{\delta}}=\sum_{|\beta|\leq m} \left\| \langle x \rangle^{\delta+\beta}\nabla^{\beta}u \right\|_{L^{\infty}}.
\end{equation*}
\end{mydef}
Note that we defined scalar functions spaces. However by setting
\begin{align*}
\l u \r_X = \sum_{i=1,\dots,d}\l u_i \r_X
\end{align*}
for $u=(u_1,\dots,u_d)$ and $X$ any previously defined space, we extend those spaces and the corresponding norms to $\R^d$-valued functions. A weighted equivalent of the Sobolev embedding holds for those spaces (see \cite{MR2473363} for a proof):
\begin{prop}\label{embedding}
Let $s,m\in\N$. If $s\geq 2$ and $\beta\leq \delta+\frac{3}{2}$, then we have the continuous embedding
\begin{equation*}
    H^{s+m}_{\delta}\xhookrightarrow{}C^{m}_{\beta}.
\end{equation*}
\end{prop}

\subsection{Minkowski vector fields}

We consider the set of Minkowski vector fields 
\begin{equation}
\left\{ \dr_\alpha, S=t\dr_t+r\dr_r,\Omega_{\alpha\beta}=x_\alpha\dr_\beta-x_\beta\dr_\alpha \right\}\label{minkowski vector fields}
\end{equation}
and denote by $Z^I$ any product of $|I|$ elements of this set, where $I$ is an 11-dimensional integer. Note that since $x_0=-x^0$, we have $\Omega_{0i}=-t\dr_i-x_i\dr_t$. These vector fields allow us to recover the usual derivatives $\dr$ while gaining a weight, whose nature explain why the derivatives $\tdr$ are called "good" derivatives:
\begin{equation}
|\dr u |\lesssim\frac{1}{1+|q|}\sum_{|I|=1}|Z^Iu|\quad\text{and}\quad|\tdr u |\lesssim \frac{1}{1+s}\sum_{|I|=1}|Z^Iu|.\label{decay minkowski}
\end{equation}
\par\leavevmode\par
The Minkowski vector fields enjoy good commutativity properties with the wave operator:
\begin{equation*}
[\Box,Z]=0
\end{equation*}
except for $S$, which satisfies $[\Box,S]=2\Box$. This implies that 
\begin{equation}\label{wave operator}
|\Box Z^I u |\lesssim\sum_{|I'|\leq |I|}|Z^{I'}\Box u|.
\end{equation}
The Minkowski vector fields behave nicely with the null forms, in the sense that the following holds:
\begin{equation}\label{null form}
|Z^IQ(\dr u , \dr v)| \lesssim \sum_{|I_1|+|I_2|\leq |I|}|Q(\dr Z^{I_1}u,\dr Z^{I_2}v)|
\end{equation}

\subsection{Useful estimates}

In this section we collect different estimates crucially used in this paper.

\subsubsection{Weighted Klainerman-Sobolev inequality} We start by an estimate which allows us to bound the $L^\infty$ norm of a function knowing the $L^2$ norm of some of its derivatives with weights.

\begin{prop}\label{prop WKNS}
Consider the weight 
\begin{equation*}
    w(q)\vcentcolon=\left\lbrace
\begin{array}{c}
1+(1+|q|)^{-\alpha}\quad\text{for}\quad q<0\\
1+(1+|q|)^{\beta}\quad\text{for}\quad q>0
\end{array}\right.
\end{equation*}
with $1<\beta<3$ and $\alpha>0$. Let $f$ be compactly supported, we have
\begin{equation}
    |w^{\frac{1}{2}}(q)f|\lesssim \frac{1}{(1+s)\sqrt{1+|q|}}\sum_{|I|\leq 3}\l w^{\frac{1}{2}}(q) Z^If\r_{L^2}.\label{inequality WKNS}
\end{equation}
\end{prop}

This is the so-called weighted Klainerman-Sobolev inequality, for which a proof can be found in \cite{MR2680391}.  The assumption on $f$ in the previous proposition can be relaxed: the proposition holds as long as the RHS of \eqref{inequality WKNS} is finite.

\subsubsection{Weighted energy estimate}

As is standard, our strategy of proof relies on energy estimates for the wave operator $\Box$.

\begin{lem}\label{weighted energy estimate}
Let $w(q)$ be a non-decreasing weight function. Then 
\begin{equation*}
    \int_{\Sigma_t}|\dr\ffi|^2 w(q)\d x+\int_0^t\int_{\Sigma_{\tau}}|\tdr\ffi|^2w'(q)\d x\d\tau\lesssim \int_{\Sigma_0}|\dr\ffi|^2 w(q)\d x+\int_0^t\int_{\Sigma_{\tau}} |\Box\ffi\dr_t\ffi|w(q) \d x\d\tau.  
\end{equation*}
\end{lem}

This lemma is the semi-linear equivalent of Lemma 6.1 in \cite{MR2680391}.  It gives a \textit{weighted} energy estimate, with a weight depending only on $q$. This is called the \textit{ghost weight} method and was first introduced by Alinhac in \cite{MR1856402}. Depending on the choice of the weight, it gives us extra decay for solutions of a wave equation in the region $q>0$, i.e the exterior of the light-cone. The other advantage of this method is the presence of a space-time integral on the LHS involving only the "good" derivatives of the solution. 

The usual energy estimate for the wave operator is obtained when $w(q)\equiv 1$, and one of the standard application of this case is the \textit{finite speed of propagation} for solutions of semi-linear wave equation.
\begin{lem}\label{speed}
We define
\begin{equation*}
C_{t_0,x_0}=\enstq{(t,x)\in\R_+\times \R^3}{0\leq t\leq t_0\quad\text{and}\quad |x-x_0|\leq t_0-t}
\end{equation*}
and $B(x_0,t_0) $ is the ball in $\R^3$ centered at $x_0$ of radius $t_0$.
Let $P:\R^{3+1}\to\R$ a smooth function such that 
\begin{equation*}
|P(X)|\lesssim |X|+|X|^2.
\end{equation*}
Let $(t_0,x_0)\in\R_+\times \R^3$ and $u:\R^{3+1}\to\R$ a solution of $\Box u=P(\dr u)$ in $C_{t_0,x_0}$. If $(u,\dr_t u)$ initially vanishes on $B(x_0,t_0)$, then $u_{|C_{t_0,x_0}}=0$.
\end{lem}

This lemma will be used to localize the different terms in the decomposition of our solution (see \eqref{ansatz}).

\section{Main result}

We now state the main result of this article.

\begin{thm}\label{premier thm}
Let $d\geq 1$, $N\geq 26$, $R>0$, $-\half<\delta < \frac{1}{2}$, $F_0\in C^{N-4}(\R^3,\R^d)$ supported in $A_0^R$ and $(\ffi_0,\ffi_1)\in H^{N+1}_{\delta}(\R^3,\R^d)\times H^{N}_{\delta+1}(\R^3,\R^d)$ such that
\begin{equation*}
    \| F_0\|_{C^{N-4}}+\l\ffi_0\r_{H^{N+1}_{\delta}}+\l\ffi_1\r_{H^{N}_{\delta+1}}\leq \e.
\end{equation*}
There exists $\e_0=\e_0(N,R,\delta)>0$ such that if $\e\leq \e_0$ and $\lambda\in(0,1]$ there exists a unique global solution $\Phi_{\lambda}:\R_+\times \R^3\longrightarrow\R^d$ to
\[
\left \{
\begin{array}{rcl}
\Box\Phi_{\lambda,i}&=&Q_i(\dr\Phi_{\lambda},\dr\Phi_{\lambda}) \\
 \Phi_{\lambda|t=0}&=&\ffi_0+\lambda F_0\cos\left( \frac{r}{\lambda}\right) \\
  \dr_t\Phi_{\lambda|t=0} & = & \ffi_1+ F_0\sin\left( \frac{r}{\lambda}\right)+\lambda\Tilde{F}_{0,\lambda}
\end{array}
\right.
\]
where $Q_i$ is a null form and $\Tilde{F}_{0,\lambda}$  is supported in $A^R_0$ depending only on $(\ffi_0,\ffi_1,F_0)$ and which satisfies $ \l \Tilde{F}_{0,\lambda}\r_{L^\infty}\lesssim \e$. The function $\Phi_{\lambda}$ admits the decomposition
\begin{equation*}
    \Phi_{\lambda} = \ffi+\lambda F\cos\left(\frac{t-r}{\lambda}\right) +\lambda^2\Tilde{F}_{\lambda},
\end{equation*}
where
\begin{itemize}
\item $\ffi\in H^{N+1}(\R_+\times\R^3,\R^d)$ satisfies $\Box\ffi_i=Q_i(\dr\ffi,\dr\ffi)$ and $(\ffi,\dr_t\ffi)_{|t=0}=(\ffi_0,\ffi_1)$,
\item $F\in C^{N-4}(\R_+\times\R^3,\R^d)$ is supported in $A^R$ and satisfies 
\begin{align*}
\left(\dr_t+\dr_r+\frac{1}{r}\right)F_i=\half \sum_{k,\ell=1,\dots,d}F_kQ_i(\dr (t-r),\dr\ffi_\ell)
\end{align*}
and $F_{|t=0}=F_0$,
\item $\Tilde{F}_\lambda$ is supported in $B^R$ and it is such that $|\Tilde{F}_\lambda|\lesssim\e$ and $|\dr\Tilde{F}_\lambda|\lesssim\frac{\e}{\lambda}$.
\end{itemize}
\end{thm}

\begin{remark}\label{limit}
The main feature of this theorem is that the smallness constant $\e_0$ does not depend on $\lambda$, which is allowed to take any value in $(0,1]$. In addition to the nature of the decomposition of $\Phi_\lambda$, this is at the heart of the high-frequency nature of the solution we construct: if we let $\lambda$ tends to 0, the high order norms of $\Phi_\lambda$ (precisely the $W^{\infty,p}$ norm for $p>1$) blows up but this is not compensated by letting $\e_0$ go to 0 as well. Therefore, the limit of $\Phi_\lambda$ as $\lambda$ tends to 0 is meaningful, and we have:
\begin{align*}
\Phi_\lambda& \to  \ffi,\quad\text{uniformly in $L^\infty$,} \\
\dr\Phi_\lambda& \rightharpoonup \dr\ffi,\quad\text{weakly in $L^2$.}
\end{align*}
\end{remark}

\section{Construction of the ansatz}\label{section Construction of the ansatz}

The heart of the proof of Theorem \ref{premier thm} is to construct a high-frequency ansatz for the solution $\Phi_\lambda$. The main purpose of this ansatz is to capture the creation of higher-order harmonics through the quadratic non-linearity. Those harmonics may also include non-oscillating terms. As explained in the introduction, we choose an ansatz with a \textit{half-chessboard} shape. More precisely we consider a solution $\Phi_\lambda$ of the form
\begin{equation}
    \Phi_{\lambda}=\ffi+ \sum_{2\leq k\leq K-1}\lambda^k\psi^{(k)} +\sum_{1\leq i\leq k \leq K}\lambda^kF^{(k,i)}\T_{\lambda}^{(k,i)}+h_{\lambda},\label{ansatz}
\end{equation}
where $K\geq 2$ is an integer and where we use the following notation for and $k,i\in\Z$: \begin{equation}
\T_{\lambda}^{(k,i)}=
\left \lbrace{ \begin{array}{cc}
\sin\left( \frac{i(t-r)}{\lambda} \right) & \textrm{ if } k+i\textrm{ is odd,} \\
\cos\left( \frac{i(t-r)}{\lambda} \right) & \textrm{ if } k+i\textrm{ is even.}
\end{array}} \right.\label{def trigo}
\end{equation}
Besides from being at the very heart of what we call an half-chessboard shaped ansatz, this notation is useful to capture the behaviours of the standard trigonometric functions regarding products and derivation (see Lemma \ref{technical lemma}). A bit of vocabulary: 
\begin{itemize}
\item $\ffi$ is called the \textit{background} solution,
\item the functions $\psi^{(\ell)}$ are called the \textit{non-oscillating} terms of the ansatz, whereas $F^{(k,i)} \T_\lambda^{(k,i)}$ are the \textit{oscillating} terms,
\item the function $h_\lambda$ is the remainder of the ansatz and is non-oscillating, however it mimics a high-frequency behaviour, i.e it loses one power of $\lambda$ for each derivatives, a fact clearly seen on the schematic estimates it will satisfy:
\begin{equation*}
\l\dr^m h_\lambda\r_{L^2}\lesssim \lambda^{K+1-m}\quad\text{for}\quad m\geq 1.
\end{equation*}
\end{itemize}

The degree of precision of the ansatz \eqref{ansatz} is given by the integer $K$. Our work is valid for all values of $K\geq 2$, but shows that global existence as a solution of \eqref{main equation} requires $K\geq 4$. As explained in the introduction, this is due to the Klainerman-Sobolev inequality.

\begin{remark}
Since we adopt the standard convention that sums over empty sets vanish,  there are no $\psi^{(\ell)}$ functions in the ansatz if $K=2$. However, one major interest of our result is precisely the freedom we have on $K$: even though we are interested in a \textit{lower} bound on $K$ to ensure the global existence of $\Phi_\lambda$, our proof also shows that we can describe the propagation of high-frequency waves with an arbitrary large precision.
\end{remark}

\begin{remark}\label{convention}
In the rest of the article we will use the following convention: $\psi^{(1)}=0$ and $F^{(k,i)}=0$ if $i>k$ or $k=0$. This will be useful in the sequel not to be too bothered with complicated sums.
\end{remark}

Our strategy of proof is to derive formally from $\Box\Phi_\lambda=Q(\dr\Phi_\lambda,\dr\Phi_\lambda)$ a system of equations for the unknowns $(\ffi,\psi^{(\ell)}, F^{(k,i)}, h_\lambda)$ and prove global existence for this system, called in the rest of this article the \textit{reduced system}. Let us start by stating a technical lemma summarizing some basic properties of the trigonometric functions $\T_\lambda^{(k,i)}$ regarding products, the wave operator and the null forms.

\begin{remark}
In the next lemma and in the rest of the article (including the technical next section), we won't write the numerical constants which appear in the computations, meaning for example that an equality like \eqref{trigo 1} has to be understood in the following sense: there exists real constants $A=A(k,\ell,i,j)$ and $B=B(k,\ell,i,j)$ such that 
\begin{equation*}
\T_{\lambda}^{(k,i)}\T_{\lambda}^{(\ell,j)}=A\T_{\lambda}^{(k+\ell,i+j)}+B\T_{\lambda}^{(k-\ell,i-j)}.
\end{equation*}
Though these numerical constants can depend on every indexes appearing in \eqref{ansatz}, they are not allowed to depend on $\lambda$, meaning that we will always keep trace of $\lambda$ in the formal computations.
\end{remark}

\begin{lem}\label{technical lemma}
Let $k,\ell,i,j\in\Z$, $f$ and $g$ any function and $Q$ any null form. We have $\T_{\lambda}^{(k,i)}=\T_{\lambda}^{(|k|,|i|)}$ and
\begin{equation}\label{trigo 1}
\T_{\lambda}^{(k,i)}\T_{\lambda}^{(\ell,j)}=\T_{\lambda}^{(k+\ell,i+j)}+\T_{\lambda}^{(k+\ell,i-j)},
\end{equation}
\begin{equation}\label{trigo 2}
\Box \left( g \T_{\lambda}^{(k,i)}\right)=\Box g\T_{\lambda}^{(k,i)}+\frac{1}{\lambda}\Ll g\T_{\lambda}^{(k-1,i)},
\end{equation}
\begin{equation}\label{trigo 3}
    Q\left(\dr\left(f\T_{\lambda}^{(k,i)}\right),\dr g \right)=Q(\dr f,\dr g)\T_{\lambda}^{(k,i)}+\frac{1}{\lambda}f\omega\tdr g\T_{\lambda}^{(k-1,i)},
\end{equation}
\begin{align}\label{trigo 4}
    Q\left(\dr\left(f\T_{\lambda}^{(k,i)}\right),\dr\left(g\T_{\lambda}^{(\ell,j)}\right) \right) & = Q(\dr f,\dr g)\T_{\lambda}^{(k,i)}\T_{\lambda}^{(\ell,j)}
    \\&\quad+\frac{1}{\lambda}f\omega\tdr g\T_{\lambda}^{(k-1,i)}\T_{\lambda}^{(\ell,j)}+\frac{1}{\lambda}g\omega\tdr f\T_{\lambda}^{(k,i)}\T_{\lambda}^{(\ell-1,j)},\nonumber
\end{align}
where $\Ll$ is the following transport operator
\begin{equation*}
\Ll\vcentcolon= \dr_t+\dr_r+\frac{1}{r}
\end{equation*}
and where $\omega\tdr$ stands for a linear combination of $\omega_i\tdr_\alpha$ for $i=1,2,3$ and $\alpha=0,1,2,3$.
\end{lem}

\begin{proof}
The proof of \eqref{trigo 1} is a direct rewriting with our notation \eqref{def trigo} of the usual linearization formulas
\begin{align*}
2\cos(a)\cos(b) & = \cos(a-b)+\cos(a+b),
\\ 2\sin(a)\cos(b) & = \sin(a-b)+\sin(a+b),
\\ 2\sin(a)\sin(b) & = \cos(a-b)-\cos(a+b).
\end{align*}
The proof of \eqref{trigo 2} follows from a direct computation, and even though $\Box$ is a second order operator, we don't lose two powers of $\lambda$ since the phase $t-r$ satisfies the eikonal equation for the Minkowski metric, i.e
\begin{align}
m^{\alpha\beta}\dr_\alpha(t-r)\dr_\beta(t-r)=0.\label{eikonal equation}
\end{align}
The proof of \eqref{trigo 3} follows from the computations
\begin{align*}
Q_0(\dr(t-r),\dr h) & = -2\tdr_0 h,
\\ Q_{0i}(\dr(t-r),\dr h) & = \tdr_i h + 2\omega_i\tdr_0 h,
\\ Q_{ij}(\dr(t-r),\dr h) & = \omega_j\tdr_i h -  \omega_i\tdr_j h ,
\end{align*}
(which holds for any scalar function $h$) and the definition of a null form as in Definition \ref{def null form}. 

The formula \eqref{trigo 4} follows directly from \eqref{trigo 3}. 
\end{proof}

\begin{remark}
Note that even though a null form only involves first order derivatives, we would expect some $\frac{1}{\lambda^2}$ terms in \eqref{trigo 4} since $Q(\dr u,\dr v)$ is a quadratic expression. It is not the case since null forms are precisely the quadratic forms vanishing on null vector fields in Minkowski space, and \eqref{eikonal equation} precisely means that the space-time gradient of $t-r$ is a null vector field.
\end{remark}

\begin{remark}
Since $|\dr^{|\alpha|}\omega_i|\lesssim 1$ in the region $\{ r\geq \frac{1}{R} \}$ for any multi-index $\alpha$ and since formulas \eqref{trigo 3} and \eqref{trigo 4} will be applied with $f$ supported in $A^R$,  the $\omega$ factors in \eqref{trigo 3} and \eqref{trigo 4} are irrelevant in our estimates and we omit them in the sequel.
\end{remark}

\subsection{High-frequency expansion of the semi-linear wave equation}

In this section, we plug the ansatz \eqref{ansatz} into the equation 
\begin{align}
\Box \Phi_\lambda = Q(\dr\Phi_\lambda,\dr\Phi_\lambda).\label{inter a}
\end{align}
The computations are rather tedious, and we present them in the form of two lemmas, one for the LHS of \eqref{inter a} and one for its RHS.

\subsubsection{Expansion of the wave operator}

We start with the RHS of \eqref{inter a}.

\begin{lem}
If $\Phi_\lambda$ is defined by \eqref{ansatz} we have
\begin{align*}
\Box \Phi_\lambda & = \sum_{0\leq p \leq K} \lambda^p \left( \Box \Phi_\lambda\right)^{(p)}
\end{align*}
with
\begin{align}
\left( \Box \Phi_\lambda\right)^{(0)} & = \Box\ffi + \Ll F^{(1,1)}\T_{\lambda}^{(0,1)}\label{Box 0}
\\ \left( \Box \Phi_\lambda\right)^{(K)} & = \sum_{1\leq i \leq K} \Box F^{(K,i)}\T_{\lambda}^{(K,i)}  + \frac{1}{\lambda^K}\Box h_{\lambda}\label{Box K}
\end{align}
and
\begin{align}
\left( \Box \Phi_\lambda\right)^{(p)} & = \sum_{1\leq i \leq p+1}\left( \Ll F^{(p+1,i)} + \Box F^{(p,i)}\right)\T_{\lambda}^{(p,i)}   + \Box\psi^{(p)} \label{Box p}
\end{align}
for $1\leq p \leq K-1$.
\end{lem}

The proof of this lemma is left to the reader and is an application of \eqref{trigo 2} together with the convention stated in Remark \ref{convention}.  

\subsubsection{Expansion of the null form}

\begin{lem}
If $\Phi_\lambda$ is defined by \eqref{ansatz} we have
\begin{align*}
Q(\dr\Phi_\lambda ,\dr\Phi_\lambda ) & = \sum_{0\leq k \leq K-1} \lambda^kQ^{(k)} + \lambda^K Q^{(\geq K)}
\end{align*}
with 
\begin{align}
Q^{(0)} & = Q(\dr\ffi ,\dr\ffi ) + F^{(1,1)}\tdr\ffi \T_\lambda^{(0,1)},\label{Q 0}
\\ Q^{(\geq K)} & = \frac{1}{\lambda^K}\mathcal{N}(h_\lambda) +  Q(\dr F,\dr\ffi)\T_{\lambda}\label{Q K}
\\&\quad + \lambda^{+} \left( F\tdr F  \T_{\lambda} + Q(\dr F,\dr F) \T_{\lambda} + Q(\dr F,\dr\psi)\T_{\lambda} + F\tdr\psi\T_{\lambda} + Q(\dr\psi,\dr\psi) \right),\nonumber
\end{align}
and
\begin{align}
&Q^{(p)} \nonumber
\\& = Q(\dr\psi^{(p)},\dr\ffi)  + \sum_{2\leq k \leq p-2} Q(\dr\psi^{(k)},\dr\psi^{(p-k)}) \label{Q p}
\\&\; + \sum_{1\leq i\leq p +1}\left( F^{(p+1,i)}\tdr\ffi  +  Q(\dr F^{(p,i)},\dr\ffi)\right)\T_{\lambda}^{(p,i)}  \nonumber
\\&\; + \sum_{\substack{1\leq i \leq k\leq p\\1\leq j \leq p-k+1}}\left( Q(\dr F^{(k,i)},\dr F^{(p-k,j)}) + F^{(k,i)}\tdr F^{(p-k+1,j)}+F^{(p-k+1,j)}\tdr F^{(k,i)}\right)\left( \T_{\lambda}^{(p,i+j)} + \T_{\lambda}^{(p,i-j)} \right)\nonumber
\\&\; +  \sum_{\substack{0\leq k \leq p-1\\1\leq i \leq k +1}}\T_{\lambda}^{(k,i)}\left( F^{(k+1,i)}\tdr\psi^{(p-k)} +Q(\dr F^{(k,i)},\dr\psi^{(p-k)})\right)\nonumber
\end{align}
for $1\leq p \leq K-1$, and where $\mathcal{N}(h_\lambda)$ is defined in \eqref{N(h)}.
\end{lem}

\begin{remark}
In this lemma, we use the following notations for clarity:
\begin{itemize}
\item $F$ or $\psi$ denotes linear combinations of terms of the families $F^{(k,i)}$ or $\psi^{(\ell)}$,
\item $\T_\lambda$ denotes linear combinations of terms of the family $\T^{(k,i)}_\lambda$,
\item in front of a linear combination, the symbol $\lambda^+$ means that the coefficients of the linear combination contains some positive powers of $\lambda$.
\end{itemize}
\end{remark}

\begin{proof}
Using \eqref{trigo 3}-\eqref{trigo 4}, we obtain by direct expansion
\begin{align}
Q(\dr\Phi_\lambda ,\dr\Phi_\lambda ) & = Q(\dr\ffi ,\dr\ffi )  + Q(\dr h_{\lambda},\dr h_{\lambda}) + Q(\dr h_{\lambda},\dr\ffi) + \lambda^+Q(\dr(F\T_{\lambda}),\dr h_{\lambda}) + \lambda^+Q(\dr\psi,\dr h_{\lambda})\nonumber
\\&\quad +\sum_{1\leq i\leq k \leq K}\lambda^k Q(\dr F^{(k,i)},\dr\ffi)\T_{\lambda}^{(k,i)} +\sum_{1\leq i\leq k \leq K}\lambda^{k-1}F^{(k,i)}\tdr\ffi\T_{\lambda}^{(k-1,i)}\nonumber
\\&\quad + \sum_{2\leq k\leq K-1}\lambda^kQ(\dr\psi^{(k)},\dr\ffi) \nonumber
\\&\quad + \sum_{\substack{1\leq k,\ell\leq K\\1\leq i\leq k\\ 1\leq j\leq \ell}}\lambda^{k+\ell-1}\left(F^{(k,i)}\tdr F^{(\ell,j)}+F^{(\ell,j)}\tdr F^{(k,i)}\right)\left( \T_{\lambda}^{(k+\ell -1,i+j)} + \T_{\lambda}^{(k+\ell - 1,i-j)} \right)\label{brutal expansion}
\\&\quad  +  \sum_{\substack{1\leq k,\ell\leq K\\1\leq i\leq k\\ 1\leq j\leq \ell}}\lambda^{k+\ell}Q(\dr F^{(k,i)},\dr F^{(\ell,j)})\left( \T_{\lambda}^{(k+\ell,i+j)}+\T_{\lambda}^{(k+\ell,i-j)} \right)\nonumber
\\&\quad +\sum_{\substack{1\leq i\leq k\leq K\\2\leq \ell\leq K-1}}\lambda^{k+\ell}Q(\dr F^{(k,i)},\dr\psi^{(\ell)})\T_{\lambda}^{(k,i)}+\sum_{\substack{1\leq i\leq k\leq K\\2\leq \ell\leq K-1}}\lambda^{k+\ell-1}F^{(k,i)}\tdr\psi^{(\ell)}\T_{\lambda}^{(k-1,i)}\nonumber
\\&\quad +\sum_{2\leq k,\ell\leq K-1}\lambda^{k+\ell}Q(\dr\psi^{(k)},\dr\psi^{(\ell)}) \nonumber
\end{align}
Note that in order to simplify the terms $F\tdr F$ we used \eqref{trigo 1} and its consequence
\begin{align*}
\T_{\lambda}^{(k-1,i)}\T_{\lambda}^{(\ell,j)} = \T_{\lambda}^{(k,i)}\T_{\lambda}^{(\ell-1,j)} = \T_{\lambda}^{(k+\ell -1,i+j)} + \T_{\lambda}^{(k+\ell - 1,i-j)}.
\end{align*}
Now we want to regroup terms by their $\lambda$ powers.  For this we use the following facts
\begin{align*}
\sum_{1\leq k,\ell\leq K}\lambda^{k+\ell}a_{k,\ell} & = \sum_{2\leq p \leq K-1}\lambda^p\sum_{1\leq k \leq p-1}a_{k,p-k} + \GO{\lambda^K}
\\ \sum_{\substack{1\leq  k\leq K\\2\leq \ell\leq K-1}}\lambda^{k+\ell} a_{k,\ell}  & = \sum_{3\leq p \leq K-1}\lambda^p \sum_{1\leq k \leq p-2}a_{k,p-k} + \GO{\lambda^K}
\\ \sum_{2\leq k,\ell\leq K-1}\lambda^{k+\ell}a_{k,\ell} & = \sum_{4\leq p \leq K-1} \lambda^p\sum_{2\leq k \leq p-2} a_{k,p-k}   + \GO{\lambda^K}
\end{align*}
which holds for every families $(a_{k,\ell})_{(k,\ell)\in\N^2}$.  Therefore if $1\leq p \leq K-1$ we have
\begin{align*}
Q^{(p)} & =  \sum_{1\leq i\leq p } Q(\dr F^{(p,i)},\dr\ffi)\T_{\lambda}^{(p,i)} + \sum_{1\leq i\leq p +1}F^{(p+1,i)}\tdr\ffi\T_{\lambda}^{(p,i)}  + Q(\dr\psi^{(p)},\dr\ffi) 
\\&\quad + \sum_{\substack{1\leq i \leq k\leq p-1\\1\leq j \leq p-k}}Q(\dr F^{(k,i)},\dr F^{(p-k,j)})\left( \T_{\lambda}^{(p,i+j)}+\T_{\lambda}^{(p,i-j)} \right) 
\\&\quad + \sum_{1\leq i \leq k\leq p-1}Q(\dr F^{(k,i)},\dr\psi^{(p-k)})\T_{\lambda}^{(k,i)} + \sum_{2\leq k \leq p-2} Q(\dr\psi^{(k)},\dr\psi^{(p-k)}) 
\\&\quad +  \sum_{\substack{1 \leq i \leq k\leq p \\ 1 \leq j \leq p-k+1}}\left(F^{(k,i)}\tdr F^{(p-k+1,j)}+F^{(p-k+1,j)}\tdr F^{(k,i)}\right)\left( \T_{\lambda}^{(p,i+j)} + \T_{\lambda}^{(p,i-j)} \right)
\\&\quad +  \sum_{1\leq i \leq k\leq p-1}F^{(k,i)}\tdr\psi^{(p-k+1)}\T_{\lambda}^{(k-1,i)}
\end{align*}
where we used several times our convention on $\psi^{(1)}$ and on sums over empty sets. If we regroup terms according to their oscillating behaviour in this expression we obtain \eqref{Q p}.

It remains to compute $Q^{(0)}$ and $Q^{(\geq K)}$. For $Q^{(0)}$, we simply look at \eqref{brutal expansion}. For $Q^{(\geq K)}$, we don't need to be precise in terms of the oscillating behaviour and the hierarchy since $\Box h_\lambda$ will absorb everything. This explains the expression of the lemma, with the following definition:
\begin{align}
\mathcal{N}(h_\lambda) = Q(\dr h_{\lambda},\dr h_{\lambda}) + Q(\dr h_{\lambda},\dr\ffi) + \lambda^+Q(\dr(F\T_{\lambda}),\dr h_{\lambda}) + \lambda^+Q(\dr\psi,\dr h_{\lambda}).\label{N(h)}
\end{align}
\end{proof}

\subsection{The reduced system}

As the two previous lemmas show, solving
\begin{equation*}
\Box\Phi_\lambda = Q(\dr\Phi_\lambda ,\dr\Phi_\lambda )
\end{equation*}
is equivalent to solving the following high-frequency hierarchy:
\begin{align}
(\Box \Phi_\lambda)^{(p)}= Q^{(p)}\label{hierarchy middle}
\end{align}
for $0\leq p \leq K-1$ and 
\begin{align}
(\Box \Phi_\lambda)^{( K)}= Q^{(\geq K)}.\label{hierarchy end}
\end{align}
Since $h_\lambda$ is not-oscillating, \eqref{hierarchy end} simply rewrites as a semi-linear wave equation with oscillating coefficients:
\begin{align*}
\Box h_\lambda & =   \mathcal{N}(h_\lambda) + \lambda^K Q(\dr F,\dr\ffi)\T_{\lambda} + \lambda^K \Box F \T_\lambda
\\&\quad +\lambda^K \lambda^{+} \left( F\tdr F  \T_{\lambda} + Q(\dr F,\dr F) \T_{\lambda} + Q(\dr F,\dr\psi)\T_{\lambda} + F\tdr\psi\T_{\lambda} + Q(\dr\psi,\dr\psi) \right)
\end{align*}

For the remaining equations of order $\lambda^p$ for $0\leq p \leq K-1$, note that both the LHS and RHS contain non-oscillating and oscillating terms, which we have to identify precisely. 

For $p=0$ it is quite straightforward: the non-oscillating part can be absorbed by the background equation
\begin{align*}
\Box \ffi=Q(\dr\ffi,\dr\ffi)
\end{align*}
and the part oscillating like $\T_\lambda^{(0,1)}$ is absorbed thanks to a transport equation for $F^{(1,1)}$:
\begin{align*}
\Ll F^{(1,1)}=F^{(1,1)}\tdr\ffi.
\end{align*}

If $1\leq p \leq K-1$, the distribution between oscillating and non-oscillating terms is less clear, especially since the interaction between the different $F^{(k,i)}$ leads to the creation of non-oscillating terms. This is a consequence of the presence of $\T_{\lambda}^{(p,i-j)}$ in \eqref{Q p}, which is non-oscillating if and only if $p$ is even and $i=j$. These created non-oscillating terms are the motivation behind the terms $\psi^{(\ell)}$, indeed they are absorbed by $\Box\psi^{(\ell)}$ present in \eqref{Box p}. Therefore, since $p$ has to be even so that $\T_{\lambda}^{(p,i-j)}$ is non-oscillating, we only need $\psi^{(\ell)}$ for $\ell$ even and we take $\psi^{(\ell)}=0$ for $\ell$ odd. 

This has the following nice consequence: in the last term of $Q^{(p)}$ (for $1\leq p \leq K-1$) 
\begin{align*}
\sum_{\substack{0\leq k \leq p-1\\1\leq i \leq k +1}}\T_{\lambda}^{(k,i)}\left( F^{(k+1,i)}\tdr\psi^{(p-k)} +Q(\dr F^{(k,i)},\dr\psi^{(p-k)})\right)
\end{align*}
we now impose that $p-k$ is even, which implies $\T_{\lambda}^{(k,i)}=\T_{\lambda}^{(p,i)}$ (since $p$ and $k$ have the same parity).  Therefore, the oscillating terms in $Q^{(p)}$ are all of the form $\T_{\lambda}^{(p,i)}$ for $1\leq i \leq p+1$ and we can rewrite $Q^{(p)}$ in the following more condensed form
\begin{align}
Q^{(p)} & = \sum_{1\leq i \leq p+1}Q^{(p,i)}\T_{\lambda}^{(p,i)} + \pi Q^{(p)} \label{Q p bis}
\end{align}
where
\begin{align*}
Q^{(p,i)} & = F^{(p+1,i)}\tdr\ffi + Q(\dr F^{(\leq p)},\dr\ffi) + Q(\dr F^{(\leq p)} , \dr F^{(\leq p)}) 
\\&\quad + F^{(\leq p)}\tdr F^{(\leq p)} + F^{(\leq p)}\tdr\psi^{(\leq p)} + Q(\dr F^{(\leq p)}, \dr \psi^{(\leq p)})
\end{align*}
and
\begin{align*}
\pi Q^{(p)} = Q(\dr \psi^{(p)},\dr\ffi) + Q(\dr \psi^{(\leq p-2)} , \dr \psi^{(\leq p-2)}) + Q(\dr F^{(\leq p)} , \dr F^{(\leq p)}) + F^{(\leq p)}\tdr F^{(\leq p)}.
\end{align*}

\begin{remark}
Here, we used the useful notation $\psi^{(\leq i)}$ to denote any $\psi^{(j)}$ for $j\leq i$.  Samewise, $F^{(\leq k)}$ denotes any $F^{(\ell,i)}$ for $1\leq i\leq\ell\leq k$.
\end{remark}

Recall that our previous remark on the presence of non-oscillating terms basically means that $\pi Q^{(p)} =0$ if $p$ is odd. Thanks to the decomposition \eqref{Q p bis}, we see that both sides of \eqref{hierarchy middle} oscillate in a similar manner, and we can assume the following equations on $\psi^{(p)}$ and $F^{(p+1,i)}$:
\begin{align*}
\Ll F^{(p+1,i)} + \Box F^{(p,i)} = Q^{(p,i)} \quad\text{and}\quad \Box\psi^{(p)} = \pi Q^{(p)}.
\end{align*}

By collecting the results of this discussion, we define the \textit{reduced system} for the unknowns $(\ffi,\psi^{(\ell)},F^{(k,i)},h_\lambda)$ to be
\begin{align}
\Box\ffi & = Q(\dr \ffi, \dr\ffi),\label{equation sur ffi}\tag{\textbf{BG}}
\\ \Ll F^{(k,i)} & = \Box F^{(k-1)} + F^{(k,i)}\tdr\ffi + Q(\dr F^{(\leq k-1)},\dr\ffi) + Q(\dr F^{(\leq k-1)} , \dr F^{(\leq k-1)}) \label{equation sur F ki}\tag{\textbf{T}$(k,i)$}
\\&\quad + F^{(\leq k-1)}\tdr F^{(\leq k-1)} + F^{(\leq k-1)}\tdr\psi^{(\leq k-1)} + Q(\dr F^{(\leq k-1)}, \dr \psi^{(\leq k-1)}),\nonumber
\\ \Box\psi^{(k)} & = Q(\dr \psi^{(k)},\dr\ffi) + Q(\dr \psi^{(\leq k-2)} , \dr \psi^{(\leq k-2)}) + Q(\dr F^{(\leq k)} , \dr F^{(\leq k)}) + F^{(\leq k)}\tdr F^{(\leq k)},\label{equation sur ffi k}\tag{\textbf{W}$(k)$}
\\ \Box h_\lambda & =   \mathcal{N}(h_\lambda) + \lambda^K Q(\dr F,\dr\ffi)\T_{\lambda} + \lambda^K \Box F \T_\lambda\label{equation sur h}\tag{\textbf{W}}
\\&\quad +\lambda^K \lambda^{+} \left( F\tdr F  \T_{\lambda} + Q(\dr F,\dr F) \T_{\lambda} + Q(\dr F,\dr\psi)\T_{\lambda} + F\tdr\psi\T_{\lambda} + Q(\dr\psi,\dr\psi) \right)\nonumber.
\end{align}
The initial data for $(\ffi,\psi^{(\ell)},F^{(k,i)},h_\lambda)$ are given by
\begin{align}\label{initial data}
    (\ffi,\dr_t\ffi)_{|t=0}&=(\ffi_0,\ffi_1),\nonumber \\
    F^{(k,i)}_{\;\;\;\quad|t=0}&=\left\{ 
\begin{array}{ll}
  F_0 & \;\; \text{if $(k,i)=(1,1)$}\\
  0 & \;\;\text{if $2\leq k\leq K$,}\\ \end{array} \right.\\
  (\psi^{(k)},\dr_t\psi^{(k)})_{|t=0}&=(0,0),\nonumber\\
  (h_{\lambda},\dr_t h_{\lambda})_{|t=0}&=(0,0).\nonumber
\end{align} 
We denote by $S$ the system of equation \eqref{equation sur ffi}-\eqref{equation sur F ki}-\eqref{equation sur ffi k}-\eqref{equation sur h} where $k$ and $i$ take all the possible values such that $F^{(k,i)}$ or $\psi^{(k)}$ are well-defined. 

\begin{remark}
Note that in this system, we follow again the convention never to write explicitly the numerical constants appearing in the expansion. But since they appear in Theorem \ref{premier thm}, let us recall the exact equations satisfied by $\ffi$ and $F^{(1,1)}$:
\begin{equation*}
\Box \ffi_i = Q_i(\dr \ffi,\dr \ffi)\quad\text{and}\quad \Ll F^{(1,1)}_i=\half \sum_{k,\ell=1,\dots,d}F^{(1,1)}_kQ_i(\dr (t-r),\dr\ffi_\ell),
\end{equation*}
where $Q_i$ are the null forms of the main system for $\Phi_\lambda$.
\end{remark}

\begin{remark}
As we explained in the discussion before the definition of the reduced system, the functions $\psi^{(k)}$ only exist for even $k$ but for the sake of clarity, it seems useful to forget this and consider that we want to solve the equations \eqref{equation sur ffi k} for all values of $k$ between 2 and $K-1$.
\end{remark}

We see that there is a coupling between \eqref{equation sur F ki} and \eqref{equation sur ffi k} but with a particular triangular structure: the RHS of \eqref{equation sur F ki} involves only $F^{(\ell+1,j)}$ and $\psi^{(\ell)}$ for $\ell\leq k-1$, and the RHS of \eqref{equation sur ffi k} involves only $F^{(\ell,j)}$ and $\psi^{(\ell)}$ for $\ell\leq k$. This is a consequence of our choice of ansatz \eqref{ansatz} and it allows us to solve this system with the following strategy:
\begin{enumerate}
\item We first solve \eqref{equation sur ffi}, this is done in Section \ref{section ffi}.
\item We solve (\textbf{T}(1,1)), which initiates a strong induction argument:
\begin{enumerate}
\item if we know $F^{(\ell,j)}$ and $\psi^{(\ell)}$ for $1\leq j\leq \ell\leq k$, we can solve $(\textbf{T}(k+1,i))$ for $1\leq i\leq k+1$...
\item ... and then solve $(\textbf{W}(k+1))$.
\end{enumerate}
This is done in Section \ref{section F ffi}.
\item Finally in Section \ref{section h} we solve \eqref{equation sur h}.
\end{enumerate}
Following this strategy, we prove the following theorem about the reduced system $S$:

\begin{thm}\label{theo principal}
Let $K\geq 4$, $N\geq 6+5K$, $R>0$, $-\half <\delta<\half$, $0<\alpha< 1$, $(\ffi_0,\ffi_1)\in H^{N+1}_{\delta}\times H^{N}_{\delta+1}$ and $F_0\in C^{N-4}$ with support in $A^R_0$ such that
\begin{equation}
    \| F_0\|_{C^{N-4}}+\l\ffi_0\r_{H^{N+1}_{\delta}}+\l\ffi_1\r_{H^{N}_{\delta+1}}\leq \e.\label{smallness assumptions}
\end{equation}
There exists $\e_0=\e_0(K,N,R,\delta)>0$ such that if $\e\leq\e_0$ and $0<\lambda\leq 1$ there exists a global solution $(\ffi,\psi^{(\ell)}, F^{(k,i)}, h_\lambda)$ to the reduced system $S$ with initial data as in \eqref{initial data}. Moreover, the following estimates hold: 
\begin{itemize}
\item $\ffi\in H^{N+1}$ satisfies 
\begin{align*}
\l w^{\frac{1}{2}}\dr Z^I\ffi\r_{L^2}&\lesssim \e\qquad \text{for}\quad|I|\leq N,
\end{align*}
with
\begin{equation*}
    w(q)\vcentcolon=\left\lbrace
\begin{array}{c}
1+(1+|q|)^{-\alpha}\quad\text{for}\quad q<0\\
1+(1+|q|)^{2(\delta+1)}\quad\text{for}\quad q>0
\end{array}\right.,
\end{equation*}
\item $F^{(k,i)}\in C^{N+4-5K}$ is supported in $A^R$ and satisfies 
\begin{equation*}
|Z^IF^{(k,i)}|\lesssim   \frac{\e}{(1+r)},\qquad\text{for}\quad |I|\leq N+4-5K,
\end{equation*}
\item $\psi^{(\ell)}\in H^{N+9-5K}$ is supported in $B^R$ and satisfies
\begin{equation*}
    \l \dr Z^I\psi^{(\ell)} \r_{L^2}\lesssim\e,\qquad\text{for}\quad |I|\leq N+8-5K,
\end{equation*}
\item $h_{\lambda}\in H^{N+3-5K}$ is supported in $B^R$ and satisfies 
\begin{equation*}
\l \dr Z^I h_\lambda\r_{L^2}\lesssim \e\lambda^{K-|I|},\qquad\text{for}\quad |I|\leq N+2-5K.
\end{equation*}
\end{itemize}
\end{thm}

\par\leavevmode\par
Note that in this article, $\l f\r_{L^2}$ denotes the $L^2$ norm on a $t=constant$ slice. The rest of this article is devoted to the proof of this theorem, but we can already see that it implies our main result, i.e Theorem \ref{premier thm}.

\begin{proof}[Proof of Theorem \ref{premier thm}]
The reduced system is defined such that if an ansatz of the form \eqref{ansatz} solves it, then $\Phi_\lambda$ solves $\Box\Phi_\lambda=Q(\dr\Phi_\lambda,\dr\Phi_\lambda)$. Therefore, Theorem \ref{theo principal} implies Theorem \ref{premier thm}, setting
\begin{align*}
\Tilde{F}_{0,\lambda}&\vcentcolon=\frac{1}{\lambda}\sum_{\substack{2\leq k\leq K\\1\leq i\leq k}}\lambda^k\dr_tF^{(k,i)}_{\quad\;\;|t=0}\T_{\lambda\quad|t=0}^{(k,i)},\\
F&\vcentcolon=F^{(1,1)},\\
\Tilde{F}_{\lambda}&\vcentcolon=\frac{1}{\lambda^2}\left(\sum_{2\leq k\leq K-1}\lambda^k\psi^{(k)}+\sum_{\substack{2\leq k\leq K\\1\leq i\leq k}}\lambda^kF^{(k,i)}\T_{\lambda}^{(k,i)}+h_\lambda\right).
\end{align*}
As mentioned in Theorem \ref{premier thm},  we want $ \Tilde{F}_{0,\lambda}$  to depend only on $(\ffi_0,\ffi_1,F_0)$. This holds since the quantities $\dr_tF^{(k,i)}_{\quad\;\;|t=0}$ can be recover through \eqref{equation sur F ki} since $\dr_t=\Ll-\dr_r-\frac{1}{r}$. Because of the triangular structure of the reduced system, it can be seen quite easily that $\dr_tF^{(k,i)}_{\quad\;\;|t=0}$ only depends on $(\ffi_0,\ffi_1,F_0)$ and their spatial derivatives.
\end{proof}

\begin{remark}\label{bad remark}
According to Klainerman's result, considering quadratic non-linearities with a null structure ensures global existence for the non-linear wave equation in space dimension 3. But from a high-frequency perspective, null forms are also of prime interest. Indeed, if $Q(\dr u,\dr v)$ were not a null form, the expansion of $Q(\dr (f \T),\dr (g\T))$ would contain a term of the form $fgQ(\dr (t-r),\dr (t-r))$. Looking at the decomposition of $Q(\dr \Phi_\lambda,\dr\Phi_\lambda)$, we see that 
the equation on $\Phi^{(1)}$ would be of the form
\begin{align}
\Ll \dr_\theta\Phi^{(1)} = \left( \dr_\theta\Phi^{(1)} \right)^2 + \cdots \label{bad equation}
\end{align}
where $\dr_\theta$ corresponds to derivatives with respect to the third variable of $\Phi^{(1)}$ (see \eqref{formal} for the formal definition of $\Phi^{(1)}$). The non-linear equation \eqref{bad equation} requires $\Phi^{(1)}$ to be described by a "full" Fourier series, i.e
\begin{align*}
\Phi^{(1)}(t,x,\theta) = \sum_{\ell\in\Z}F^{(1,\ell)}(t,x)\exp\left(i\ell \theta \right)
\end{align*}
with an infinite amount of $F^{(1,\ell)}$ non-zero. Though the equations for the higher order terms (i.e $\Phi^{(i)}$ for $i\geq 1$) would remain linear equations, the triangular structure and the presence of $\Phi^{(1)}$ as source term would also imply that all the $\Phi^{(i)}$ are described by full Fourier series.  Therefore the null structure allows us to consider a "simple" ansatz, the so-called half-chessboard ansatz.
\end{remark}

\begin{remark}\label{remark poids}
The weight $w$ is parametrized by its exponents in the exterior region $q>0$ and in the interior region $q<0$. The exterior exponent $2(\delta+1)$ is chosen so that the bound $\l w^\half \dr Z^{I}\ffi\r_{L^2}\lesssim \e$ holds initially and follows from the smallness assumptions \eqref{smallness assumptions}. The constraint on $\delta$ then implies that $1<2(\delta+1)<3$, which will allow us to apply Proposition \ref{prop WKNS}. The interior exponent $\alpha$ is chosen so that there exists $0<\nu\leq 1$ and $2-\nu>\alpha+1$ (take for example $\nu=\frac{1-\alpha}{2}$), the latter implying
\begin{align}
\frac{w}{(1+|q|)^{2-\nu}}\lesssim w'.\label{w1 w1'}
\end{align}   
\end{remark}

\section{Global existence for the reduced system}\label{section 5}

In this section, we prove Theorem \ref{theo principal} following the strategy outlined previously.

\subsection{The background wave equation}\label{section ffi}

We first study the global existence for \eqref{equation sur ffi}. From the celebrated work of Klainermann on null forms, we know that a global solution exists. However, since the resolution of \eqref{equation sur h} will be a sort of high-frequency equivalent of this one, it seems helpful to the author to write down the complete argument leading to the global existence for $\eqref{equation sur ffi}$. The proof is a bootstrap argument, based on both $L^\infty$ and $L^2$ estimates. As mentioned in the introduction, we use Alinhac ghost weight method. Since the equation we solve satisfy the null condition, this is not mandatory. However, the ghost weight method leads to a very efficient proof.

\begin{prop}\label{prop ffi}
There exists $\e_0>0$ such that, if $\e\leq \e_0$, \eqref{equation sur ffi} admits a unique global solution $\ffi$. Moreover, $\ffi$ satisfies 
\begin{align}
 \l w_1^{\frac{1}{2}}\dr Z^I\ffi\r_{L^2}&\lesssim \e\qquad\text{for}\quad|I|\leq N.\label{L2 ffi}
\end{align}

\end{prop}

\begin{proof}
We define $T_0$ to be the maximal time of existence of a solution $\ffi$, which, by the local in time theory, is positive and satisfies the following blow-up criterion:
\begin{equation*}
    T_0<+\infty \iff \lim_{t\to T_0^-}\sum_{|I|\leq N}\l w_1^{\frac{1}{2}}\dr Z^I\ffi\r_{L^2}(t)=+\infty.
\end{equation*}
Let $C_0\geq 1$ be a constant to be choosen later. We define $T<T_0$ to be the maximal time such that the inequality
\begin{align}
    \l w_1^{\frac{1}{2}}\dr Z^I\ffi\r_{L^2}(t)&\leq C_0\e\qquad\qquad\qquad\quad\text{for}\quad|I|\leq N,\label{BA u 2}
\end{align}
hold for $0\leq t\leq T$. By the assumptions on $\ffi_0$ and $\ffi_1$ and the Proposition \ref{embedding}, we have $T>0$, for $C_0$ large enough.

\par\leavevmode\par

We start by deriving decay for $Z^I\ffi$ from \eqref{BA u 2}. If $|I|\leq N-3$ we can apply the Klainerman-Sobolev inequality stated in Proposition \ref{prop WKNS} (since $1<2(\delta+1)<3$ and $\alpha>0$), use the bootstrap assumption \eqref{BA u 2} and obtain
\begin{align*}
\left| w_1^\half \dr Z^I\ffi \right| & \lesssim \frac{C_0\e}{(1+s)\sqrt{1+|q|}}.
\end{align*}
Using the behaviour of the weight $w_1$ in the exterior and interior region this implies
\begin{equation}\label{decay Z I ffi}
\left| \dr Z^I\ffi \right|  \lesssim \left\{
\begin{array}{c}
\frac{C_0\e}{(1+s)\sqrt{1+|q|}}  \quad\text{for}\quad q<0\\
\frac{C_0\e}{(1+s) (1+|q|)^{\delta+\frac{3}{2}} }   \quad\text{for}\quad q>0
\end{array}\right..
\end{equation}
In order to derive decay for $Z^I\ffi$, we integrate \eqref{decay Z I ffi} along the constant $s$ lines. Recall that the assumptions \eqref{smallness assumptions} and the weighted Sobolev embedding stated in Proposition \ref{embedding} implies that \[\left| \dr Z^I\ffi \right| \lesssim \frac{ C_0\e } { (1+r)^{\delta + \frac{3}{2}}}\] on the initial hypersurface. Note that $r=s$ if $t=0$. Now, let $q\geq 0$ and integrate \eqref{decay Z I ffi} from the initial hypersurface along the constant $s$ lines, this gives
\begin{align*}
\left|Z^I\ffi(q\geq 0)\right| & \lesssim \frac{ C_0\e } { (1+s)^{\delta + \frac{3}{2}}}  + \int_q^s \frac{C_0\e \;\d q'}{ (1+s)(1+|q'|)^{\delta+\frac{3}{2}}   } 
\\& \lesssim \frac{ C_0\e } { (1+s)^{\delta + \frac{3}{2}}} + \frac{ C_0\e } { (1+s)(1+|q|)^{\delta + \frac{1}{2}}}
\\& \lesssim  \frac{ C_0\e } { (1+s)(1+|q|)^{\delta + \frac{1}{2}}}
\end{align*}
where we also used the initial spatial decay and $|q|\leq s$. In particular, we obtain the following decay along the null cone $\left|Z^I\ffi(q = 0)\right| \lesssim \frac{C_0\e}{1+s}$. Now, let $q<0$ and integrate \eqref{decay Z I ffi} from the null cone along the constant $s$ lines, this gives
\begin{align*}
\left|Z^I\ffi(q < 0)\right| & \lesssim \frac{C_0\e}{1+s} + \int_q^0 \frac{C_0\e \; \d q'}{(1+s)\sqrt{1+|q'|}} \lesssim C_0\e \frac{\sqrt{1+|q|}}{1+s}.
\end{align*}
Since $\delta+\half>0$, the worst estimate is the interior one and overall we have proved that 
\begin{equation}
| Z^I\ffi|\lesssim \frac{C_0\e\sqrt{1+|q|}}{(1+s)}, \qquad \text{for}\quad |I|\leq N-3.  \label{WKNS}
\end{equation}

\par\leavevmode\par
We then improve the inequality \eqref{BA u 2}, using the weighted energy estimate given by Lemma \ref{weighted energy estimate}. We set 
\begin{align*}
E(t)&\vcentcolon=\sum_{|I|\leq N}\l w_1^{\frac{1}{2}}\dr Z^I\ffi\r^2_{L^2}(t),\\
S(t)&\vcentcolon=\int_0^t\sum_{|I|\leq N}\l (w'_1)^{\frac{1}{2}}\tdr Z^I\ffi\r_{L^2}^2(\tau)\d\tau.
\end{align*}
We use Lemma \ref{weighted energy estimate}, sum for $|I|\leq N$ and thanks to \eqref{wave operator} and \eqref{null form} we obtain for $t\in[0,T]$:
\begin{align}
    E(t)+&S(t) \lesssim E(0)+\sum_{|I_1|+|I_2|\leq N}\int_0^t\l w_1\dr Z^{I_1}\ffi \tdr Z^{I_2}\ffi \dr_tZ^I\ffi\r_{L^1}(\tau)\d\tau\label{WEE}.
\end{align}
Because of the assumption $(\ffi_0,\ffi_1)\in H^{N+1}_\delta\times H^N_{\delta+1}$ and the choice of exponent in the exterior region for the weight $w_1$, we have $E(0)\lesssim\e^2$. Since $|I_1|+|I_2|\leq N$, it must holds that $|I_2|\leq\frac{N}{2}$ or $|I_1|\leq\frac{N}{2}$. We examine the two cases separately:
\begin{itemize}
    \item If $|I_1|\leq \frac{N}{2}$, then 
    \begin{equation*}
        |\dr Z^{I_1}\ffi|\lesssim \frac{1}{1+|q|}|Z^{I_1+1}\ffi|\lesssim \frac{C_0\e}{(1+s)^{\half +\frac{\nu}{2}}(1+|q|)^{1-\frac{\nu}{2}}}
    \end{equation*}
    where we used \eqref{decay minkowski} and \eqref{WKNS} and where $\nu$ is as in the Remark \ref{remark poids}. Using $2ab\leq a^2+b^2$ and \eqref{BA u 2} we obtain:
    \begin{align*}
        &\l w_1\dr Z^{I_1}\ffi \tdr Z^{I_2}\ffi \dr_tZ^I\ffi\r_{L^1}(\tau)
        \\&\qquad\qquad\lesssim \frac{C_0\e}{(1+\tau)^{1+\nu}}\l w_1^{\frac{1}{2}}\dr_tZ^I\ffi\r_{L^2}^2(\tau)+C_0\e\l \frac{w_1^{\frac{1}{2}}}{(1+|q|)^{1-\frac{\nu}{2}}}\tdr Z^{I_2}\ffi\r_{L^2}^2(\tau)
         \\&\qquad\qquad\lesssim \frac{C_0^3\e^3}{(1+\tau)^{1+\nu}}+C_0\e\l  \left( w'_1 \right)^\half  \tdr Z^{I_2}\ffi\r_{L^2}^2(\tau).
    \end{align*}
    The last step is valid because of \eqref{w1 w1'}.
    \item If $|I_2|\leq \frac{N}{2}$, then
    \begin{equation*}
        |\tdr Z^{I_2}\ffi|\lesssim \frac{1}{1+s}|Z^{I_2+1}\ffi|\lesssim \frac{C_0\e}{(1+s)^{\frac{3}{2}}},
    \end{equation*}
     where we agin used \eqref{decay minkowski} and \eqref{WKNS}. Using the Cauchy-Schwarz inequality and \eqref{BA u 2} we obtain directly:
     \begin{align*}
     \l w_1\dr Z^{I_1}\ffi \tdr Z^{I_2}\ffi \dr_tZ^I\ffi\r_{L^1}(\tau)&\lesssim \frac{C_0\e}{(1+\tau)^{\frac{3}{2}}} \l w_1^\half \dr Z^{I_1}\ffi \r_{L^2} \l w_1^\half \dr_t Z^{I}\ffi \r_{L^2} \lesssim \frac{C_0^3\e^3}{(1+\tau)^{\frac{3}{2}}}
     \end{align*}
\end{itemize}
Therefore,  since $t\longmapsto(1+\tau)^{-1-\nu}$ and $t\longmapsto(1+\tau)^{-\frac{3}{2}}$ are both integrable over $\R_+$ we obtain for $t\in [0,T]$:
\begin{align*}
\sum_{|I_1|+|I_2|\leq N} \int_0^t \l w_1\dr Z^{I_1}\ffi \tdr Z^{I_2}\ffi \dr_tZ^I\ffi\r_{L^1}(\tau)\d\tau \lesssim C_0^3\e^3 + C_0\e S(t).
\end{align*}
Absorbing $C_0\e S(t)$into the LHS of \eqref{WEE} we have proved that 
\begin{align*}
\l w_1^{\frac{1}{2}}\dr Z^I\ffi\r_{L^2}&\leq C\e+C\sqrt{C_0\e}C_0\e\qquad\text{for}\quad|I|\leq N,
\end{align*}
with $C>0$ a numerical constant. We now choose $C_0$ such that $C\leq \frac{C_0}{4}$ and $\e_0$ such that $CC_0\e\leq \frac{C_0}{4}$ and $C\sqrt{C_0\e}\leq\frac{1}{4}$. Thus, we proved that the inequality \eqref{BA u 2} hold with a constant $\frac{C_0}{2}$ instead of $C_0$, which contradict the maximality of $T$, thus proving that $T=T_0$. But if $T=T_0$, it implies that the energy is bounded up to the time $T_0$, implying that $T_0=+\infty$, which concludes the proof.
\end{proof}

\subsection{The high-frequency hierarchy}\label{section F ffi}

In this section, we solve the coupled transport equations \eqref{equation sur F ki} and wave equations \eqref{equation sur ffi k} with a strong induction argument which heavily relies on the triangular structure of the reduced system.

The equations \eqref{equation sur F ki} are of the form  
\begin{equation}\label{Ll1}
\Ll f = f\mu +g.
\end{equation}
Everything we need to know about \eqref{Ll1} is contained in the following proposition, whose proof is postponed to Appendix \ref{appendice A}.

\begin{prop}\label{equation de transport prop}
Let $M\in\N$, $f_0:\R^3\to\R$ a $C^M$ function supported in $A_0$ such that
\begin{equation*}
\l f_0\r_{C^M}  \lesssim \e.
\end{equation*}
Moreover, let $\mu,g:\R_+\times\R^3\to\R$ two $C^M$ functions with $g$ supported in $A^R$ and such that 
\begin{align*}
(1+r)^2|Z^I\mu|+(1+r)^3|Z^Ig|&\lesssim \e \qquad \text{for}\quad |I|\leq M.
\end{align*}
There exists a unique global solution $f\in C^M$ to \eqref{Ll1} supported in $A^R$ and satisfying
\begin{equation*}
|Z^If|\lesssim \frac{\e}{1+r} \qquad \text{for}\quad |I|\leq M.
\end{equation*}
\end{prop}

With Proposition \ref{equation de transport prop} we are now ready to solve \eqref{equation sur F ki} and \eqref{equation sur ffi k}. Since we will prove that the functions $F^{(k,i)}$ are supported in $A^R$, we will be only interested in the decay in $r$, which gives decay in $s$. We will also forget about $q$ and use many times the fact that in the region $A^R$, $r$ and $t$ are equivalent. We prove the following proposition:

\begin{prop}\label{prop récurrence}
Let $1\leq i\leq  k\leq K$ and $2\leq \ell\leq K-1$. There exists $F^{(k,i)}\in C^{N_k}$ supported in $A^R$ solving \eqref{equation sur F ki} and $\psi^{(\ell)}\in H^{M_\ell}$ supported in $B^R$ solving \eqref{equation sur ffi k}.Moreover, they satisfy the following estimates:
\begin{align}
 | Z^I F^{(k,i)} | & \lesssim \frac{\e}{1+r}, \quad\qquad\text{for}\quad |I|\leq N_k,\label{A}
\\ | Z^I \Box F^{(k,i)} | & \lesssim \frac{\e}{(1+r)^3}, \qquad\text{for}\quad |I|\leq N_k-2,\label{B}
\\ \l \dr Z^I \psi^{(\ell)} \r_{L^2} & \lesssim \e, \quad\qquad\qquad\text{for}\quad |I|\leq M_\ell .\label{C}
\end{align} 
where $N_k=N+4-5k$ for $k\geq 2$ and $N_1=N-4$ and $M_\ell = N+3-5\ell$.
\end{prop}

\begin{remark}
Since $\psi^{(1)}$ is not defined, the limiting term in the equation $\mathbf{T}(2,i)$ is $\Box F^{(1,1)}$, whereas if $k\geq 3$ then the limiting term in the equation $\mathbf{T}(k,i)$ is $\dr\psi^{(k-1)}$ since we lose three derivatives using the Klainerman-Sobolev inequality to obtain decay from $L^2$ estimates. This explains why $N_1-N_2 \neq N_k - N_{k+1}$ for all $k\geq 2$.
\end{remark}

To prove Proposition \ref{prop récurrence}, we proceed by strong induction on the value of $k$.  More precisely, we will show that if the estimates \eqref{A}-\eqref{B}-\eqref{C} holds for $1\leq k \leq k'$ and $2\leq \ell\leq k'$ for $k'$ some integer satisfying $2\leq k'\leq K-2$, then they also hold for $k,\ell=k'+1$.  Note that the borderline cases $k=1$ and $k=K$, for which $\psi^{(k)}$ is not defined, are proved in a similar manner so we don't write down their specific proofs.

\par\leavevmode\par
Before we start, we derive decay for solutions of wave equation, i.e $\psi^{(\ell)}$ for $\ell\leq k'$ and $\ffi$. This follows from the weighted Klainerman-Sobolev inequality of Proposition \ref{prop WKNS} and the exact argument has already been given in the proof of Proposition \ref{prop ffi}. Following the same strategy, we obtain from \eqref{L2 ffi} and \eqref{C}:
\begin{align}
|Z^I\psi^{(\ell)}|&\lesssim \e\frac{\sqrt{1+|q|}}{1+s},\qquad\text{for}\quad |I|\leq M_\ell-3 \quad\text{and}\quad \ell\leq k', \label{decay ffi l}\\
|Z^I\ffi|&\lesssim \e\frac{\sqrt{1+|q|}}{1+s},\qquad\text{for}\quad |I|\leq N-3. \label{decay ffi}
\end{align}

\subsubsection{The transport equation for the oscillating terms} 

We start by solving $(\textbf{T}(k'+1,i))$, which we recall:
\begin{align*}
\Ll F^{(k'+1,i)} & = \Box F^{(k')} + F^{(k'+1,i)}\tdr\ffi + Q(\dr F^{(\leq k')},\dr\ffi) + Q(\dr F^{(\leq k')} , \dr F^{(\leq k')}) 
\\&\quad + F^{(\leq k')}\tdr F^{(\leq k')} + F^{(\leq k')}\tdr\psi^{(\leq k')} + Q(\dr F^{(\leq k')}, \dr \psi^{(\leq k')}).\nonumber
\end{align*}
In sake of clarity we define $G^{(k'+1,i)}$ such that $\Ll F^{(k'+1,i)} = F^{(k'+1,i)}\tdr\ffi + G^{(k'+1,i)}$. Note that the expression of $G^{(k'+1,i)}$ involves only $F^{(\ell,j)}$ for $\ell\leq k'$ and $\ffi^{(\ell)}$ for $\ell\leq k'$ so we can estimate it using \eqref{A}, \eqref{B} and \eqref{C}. This is done in the following lemma.

\begin{lem}\label{estimee sur G k+1}
$G^{(k'+1,i)}$ is supported in $A^R$ and the following estimate holds 
\begin{equation*}
|Z^IG^{(k'+1,i)}|\lesssim \frac{\e}{(1+r)^{3}},\qquad\text{for}\quad |I|\leq N-1-5k',
\end{equation*}
\end{lem}

\begin{proof}
Formally we have
\begin{align*}
G^{(k'+1,i)} & = \Box F^{(k')} + Q(\dr F^{(\leq k')},\dr\ffi) + Q(\dr F^{(\leq k')} , \dr F^{(\leq k')}) + Q(\dr F^{(\leq k')}, \dr \psi^{(\leq k')}).
\end{align*}
Let us count how many times each terms can be differentiated:
\begin{itemize}
\item we want to estimate $\dr\psi^{(\leq k')}$ using \eqref{decay ffi l}, we can differentiate this term at most $M_{k'}-4$ times,
\item the terms $\dr F^{(\leq k')}$ are better than $\Box F^{(k')}$ which thanks to \eqref{B} can be differentiated at most $N_{k'}-2$ times.
\end{itemize}
Since $M_{k'}-4<N_{k'}-2$, we can differentiate $G^{(k'+1,i)}$ at most $M_{k'}-4$ times, i.e at most $N-1-5k'$ times. Now,  let $|I|\leq N-1-5k'$, we have:
\begin{align*}
|Z^IG^{(k'+1,i)}| & \lesssim |Z^I\Box F^{(k',i)}|+ \sum_{|I_1|+|I_2|\leq|I|+1}\frac{|Z^{I_1}F^{(\leq k')}||Z^{I_2}F^{(\leq k')}|}{1+r}
 \\&\qquad+\sum_{|I_1|+|I_2|\leq |I|+1}\frac{|Z^{I_1+1}F^{(\leq k')}||Z^{I_2+1}\ffi|}{1+r} + \sum_{|I_1|+|I_2|\leq|I|+1}\frac{|Z^{I_1}F^{(\leq k')}||Z^{I_2}\psi^{(\leq k')}|}{1+r} 
 \\&\lesssim \frac{\e}{(1+r)^3},
\end{align*}
where we used \eqref{A}, \eqref{B}, \eqref{decay ffi l} and \eqref{decay ffi}.
\end{proof}

\begin{lem}
There exists a unique solution $F^{(k'+1,i)}\in C^{N-1-5k'}$ to $(\textbf{T}(k'+1,i))$ supported in $A^R$ such that
\begin{align}
    |Z^IF^{(k'+1,i)}|&\lesssim   \frac{\e}{(1+r)},\qquad\text{for}\quad |I|\leq N-1-5k',\label{estimee sur F k+1}\\
    |Z^I\Box F^{(k'+1,i)}|&\lesssim  \frac{\e}{(1+r)^{3}},\qquad\text{for}\quad |I|\leq N-3-5k'.\label{estimee sur box F k+1}
\end{align}
\end{lem}

\begin{proof}
We apply Proposition \ref{equation de transport prop} with $M=N-1-5k'$, $g=G^{(k'+1,i)}$ and $\mu=\tdr\ffi$ which satisfy the required estimates thanks to Lemma \ref{estimee sur G k+1} and to
\begin{equation}\label{estimation Lu}
    \left| Z^I\tdr\ffi\right|\lesssim |\tdr Z^I\ffi|+\sum_{|J|\leq|I|}\frac{|Z^J\ffi|}{1+s}\lesssim\sum_{|J|\leq|I|+1}\frac{| Z^J \ffi|}{1+s} \lesssim \frac{\e}{(1+r)^2}
\end{equation}
in the region $A^R$ for $|I|\leq N-4$ and where we use \eqref{decay ffi}. This concludes the proof of \eqref{estimee sur F k+1}. Let us now turn to the proof of \eqref{estimee sur box F k+1}. Because of Lemma \ref{d'alembertien}, we have 
\begin{equation*}
    \Box F^{(k'+1,i)}=\delta^{ij}\tdr_i\tdr_jF^{(k'+1,i)}+\underline{L}(F^{(k'+1,i)}\tdr\ffi)+\frac{1}{r}F^{(k'+1,i)}\tdr\ffi+\underline{L}G^{(k'+1,i)}+\frac{1}{r}G^{(k'+1,i)},
\end{equation*}
where $\underline{L}=\dr_t-\dr_r$. We estimate this expression directly, for $I$ a multi-index such that $|I|\leq N-3-5k'$:
\begin{align*}
    |Z^I\Box F^{(k'+1,i)}|&\lesssim |Z^I\tdr^2F^{(k'+1,i)}|+|Z^I\dr(F^{(k'+1,i)}\tdr\ffi)|+|Z^I(r^{-1}F^{(k'+1,i)}\tdr\ffi)|\\&\qquad+|Z^I\dr G^{(k'+1,i)}|+|Z^I(r^{-1}G^{(k'+1,i)})|\\
    &\lesssim \sum_{|J|\leq |I|+2}\frac{|Z^JF^{(k'+1,i)}|}{(1+r)^2}+  \sum_{|I_1|+|I_2|\leq|I|+1}|Z^{I_1}F^{(k'+1,i)}||Z^{I_2}\tdr\ffi| + \sum_{|J|\leq |I|+1}|Z^JG^{(k'+1,i)}| 
    \\&\lesssim  \frac{\e}{(1+r)^{3}}
\end{align*}
where we used commutation property of $Z^I$ and $\tdr$ as in \eqref{estimation Lu}, the Leibniz rule and $|Z^K(r^{-1})|\lesssim r^{-1}$ for every multi-index $K$. Using \eqref{estimee sur F k+1}, \eqref{estimation Lu} and Lemma \ref{estimee sur G k+1} we conclude the proof of \eqref{estimee sur box F k+1}.
\end{proof}

\begin{remark}\label{remarque 5.2}
Note that in the previous proof, we used Lemma \ref{d'alembertien}, which basically shows that if we have a good control on $\Ll f$, then $\Box f$ decays better than two derivatives of $f$, since it behaves like two "good" derivatives of $f$, i.e $\tdr^2 f$. This prevents us from estimating $\Box f$ by commuting $\Box$ and $\Ll$. The decay obtained through this method would not be sufficient since 
\begin{equation*}
[\Ll,\Box]f=-\frac{2}{r}\delta^{ij}\dr_i\tdr_j f + \text{l.o.t}
\end{equation*}
where "l.o.t" stands for lower order terms. This would imply that if $f$ solves \eqref{Ll1}, then $\Box f$ solves
\begin{equation*}
(L-\mu)(r\Box f) = \dr\tdr f + \text{l.o.t}.
\end{equation*}
Since a normal derivative doesn't give extra decay in $A^R$, we would only get $\Box f\sim \frac{1}{r}$ while we need $\frac{1}{r^3}$, given Proposition \ref{equation de transport prop}.
\end{remark}

\subsubsection{The wave equation for the non-oscillating terms}

We now solve $(\textbf{W}(k'+1))$:
\begin{align*}
\Box\psi^{(k'+1)} & = Q(\dr \psi^{(k'+1)},\dr\ffi) + Q(\dr \psi^{(\leq k'-1)} , \dr \psi^{(\leq k'-1)}) 
\\&\quad + Q(\dr F^{(\leq k'+1)} , \dr F^{(\leq k'+1)}) + F^{(\leq k'+1)}\tdr F^{(\leq k'+1)}.
\end{align*}
We define $H^{(k'+1)}$ such that $\Box\psi^{(k'+1)}=Q(\dr\psi^{(k'+1)},\dr\ffi)+H^{(k'+1)}$. Note that $H^{(k'+1)}$ involves $\psi^{(\ell)}$ for $\ell\leq k'-1$ and $F^{(\ell,i)}$ for $\ell\leq k'+1$, thus respecting the triangular structure. We estimate $H^{(k'+1)}$ in the following lemma.

\begin{lem}\label{lem H 2k+2}
$H^{(k'+1)}$ is supported in $B^R$ and the following estimate holds 
\begin{equation}
|Z^I H^{(k'+1)}|\lesssim \frac{\e^2}{(1+s)^3} \qquad \text{for}\quad |I|\leq N-2-5k'.\label{H 2k+2}
\end{equation}
\end{lem}

\begin{proof}
Formally we have 
\begin{align*}
H^{(k'+1)} & = Q(\dr \psi^{(\leq k'-1)} , \dr \psi^{(\leq k'-1)}) + Q(\dr F^{(\leq k'+1)} , \dr F^{(\leq k'+1)}) + F^{(\leq k'+1)}\tdr F^{(\leq k'+1)}.
\end{align*}
Let us count how many times each terms can be differentiated:
\begin{itemize}
\item we want to estimate $\dr\psi^{(\leq k'-1)}$ using \eqref{decay ffi l}, we can differentiate this term at most $M_{k'-1}-4$ times,
\item we want to estimate $\dr F^{(\leq k'+1)}$ using \eqref{estimee sur F k+1} for $\dr F^{(k'+1)}$ or \eqref{A} for $\dr F^{(\leq k')}$ so we can differentiate these terms at most $N-2-5k'$.
\end{itemize}
Since $N-2-5k'<M_{k'-1}-4$, we can differentiate $H^{(k'+1)}$ at most $N-2-5k'$. Now, let $|I|\leq N-2-5k'$, we estimate $Z^IH^{(k'+1)}$ using \eqref{A}, \eqref{estimee sur F k+1} and \eqref{decay ffi l}:
\begin{align*}
|Z^I H^{(2k+2)}| & \lesssim \sum_{|I_1|+|I_2|\leq |I|}\left(  \frac{|Z^{I_1+1}F^{(\leq k'+1)}||Z^{I_2+1}F^{(\leq k'+1)}|}{1+r}+ \frac{|Z^{I_1+1}\psi^{(\leq k'-1)}||Z^{I_2+1}\psi^{(\leq k'-1)}|}{(1+s)(1+|q|)} \right)
\\& \lesssim \frac{\e^2}{(1+s)^3}.
\end{align*}
\end{proof}

\begin{lem}\label{lem psi k'+1}
If $\e$ is small enough, $(\textbf{W}(k'+1))$ admits a unique global solution $\psi^{(k'+1)}\in H^{N-2-5k'}$ supported in $B^R$ and satisfying
\begin{align*}
    \l \dr Z^I\psi^{(k'+1)}\r_{L^2}\lesssim \e\qquad\text{for}\quad|I|\leq N-2-5k'.
\end{align*}
\end{lem}

\begin{proof}
In order to prove Lemma \ref{lem psi k'+1} we use a continuity argument. We let $T_0$ the maximal time of existence of a solution $\psi^{(k'+1)}$ to $(\textbf{W}(k'+1))$, which satisfies 
\begin{equation*}
    T_0<+\infty \iff \lim_{t\to T_0^-}\sum_{|I|\leq N-2-5k'}\l w^{\frac{1}{2}}\dr Z^I\psi^{(k'+1)}\r_{L^2}(t)=+\infty
\end{equation*}
where $w$ is defined in Theorem \ref{theo principal}. Because of the support properties of $H^{(k'+1)}$ and Lemma \ref{speed}, the function $\psi^{(k'+1)}$ is supported in $B_R$ (therefore the exterior exponent of $w$ plays no role here). We make the following bootstrap assumption:
\begin{equation}\label{bootstrap sur ffi 2k+2}
\l w^{\frac{1}{2}}\dr Z^I\psi^{(k'+1)}\r_{L^2}(t)\leq \e\qquad\text{for}\quad|I|\leq N-2-5k'.
\end{equation}
Since the initial data for $\psi^{(k'+1)}$ vanish, we can assume that \eqref{bootstrap sur ffi 2k+2} holds on some time interval $[0,T]$ with $T>0$. Moreover, we assume that $T$ is the maximal time so that \eqref{bootstrap sur ffi 2k+2} holds. We define
\begin{align*}
E_{k'+1}(t)&\vcentcolon=\sum_{|I|\leq N-2-5k'}\l w^{\frac{1}{2}}\dr Z^I\psi^{(k'+1)}\r^2_{L^2}(t),\\
S_{k'+1}(t)&\vcentcolon=\int_0^t\sum_{|I|\leq  N-2-5k'}\l (w')^{\frac{1}{2}}\tdr Z^I\psi^{(k'+1)}\r_{L^2}^2(\tau)\d\tau.
\end{align*}
We apply the weighted energy estimate of Lemma \ref{weighted energy estimate},  the Cauchy-Schwarz inequality and our bootstrap assumption \eqref{bootstrap sur ffi 2k+2} to obtain (recall that the initial data for $\psi^{(k'+1)}$ vanish):
\begin{align}
E_{k'+1}(t)+S_{k'+1}(t) & \lesssim \sum_{|I|\leq N-2-5k'} \int_0^t \l w\Box Z^I \psi^{(k'+1)} \dr Z^I\psi^{(k'+1)} \r_{L^1(B_R\cap \Sigma_\tau)} \d\tau\nonumber
\\&\lesssim \sum_{|I|\leq N-2-5k'} \int_0^t \l w^\half  \Box Z^I \psi^{(k'+1)}  \r_{L^2(B_R\cap \Sigma_\tau)}\l w^\half \dr Z^I\psi^{(k'+1)} \r_{L^2} \d\tau. \label{energie psi k'+1}
\end{align}
Now if $|I|\leq N-2-5k'$, using \eqref{wave operator} and \eqref{null form} we obtain:
\begin{align*}
 \l w^\half  \Box Z^I \psi^{(k'+1)}  \r_{L^2(B_R\cap \Sigma_\tau)}  &\lesssim  \sum_{|J_1|+|J_2|\leq|I|}  \l w^\half \dr Z^{J_1}\psi^{(k'+1)} \tdr Z^{J_2}\ffi  \r_{L^2(B^R\cap\Sigma_\tau)}
\\&\quad + \sum_{|J_1|+|J_2|\leq |I|}  \l w^\half \tdr Z^{J_1}\psi^{(k'+1)} \dr Z^{J_2}\ffi  \r_{L^2(B^R\cap\Sigma_\tau)}
\\&\quad + \sum_{|J|\leq  |I|} \l w^\half Z^JH^{(k'+1)} \r_{L^2(B^R\cap\Sigma_\tau)}
\\&=\vcentcolon A + B + C.
\end{align*}
We start with the easiest term, i.e $C$, for which we use \eqref{H 2k+2} and the fact that $w$ is bounded on $B_R$:
\begin{align*}
C &\lesssim \e^2 \left( \int_0^{\tau +R} \frac{\d r}{(1+r+\tau)^4} \right)^\half \lesssim \frac{\e^2}{(1+\tau)^{\frac{3}{2}}}
\end{align*}
which implies $C\lesssim \e^2$.  For $A$,  we use \eqref{decay ffi} and $|q|\leq s$ to get
\begin{equation*}
|\tdr Z^{J_2}\ffi|\lesssim \frac{\e}{(1+s)^{\frac{3}{2}}}
\end{equation*}
for all $|J_2|\leq N-2-5k'$.This gives
\begin{align*}
A & \lesssim \e \sum_{|I|\leq N-2-5k'} \frac{1}{(1+\tau)^{\frac{3}{2}}} \l w^\half \dr Z^{I}\psi^{(k'+1)}\r_{L^2}  \lesssim \frac{\e^2}{(1+\tau)^{\frac{3}{2}}}
\end{align*}
where we also used \eqref{bootstrap sur ffi 2k+2}. Let us now look at $B$, by first using \eqref{decay ffi} to obtain
\begin{equation*}
|\dr Z^{J_2}\ffi|\lesssim \frac{\e}{(1+s)^{\half+\frac{\nu}{2}}(1+|q|)^{1-\frac{\nu}{2}}}
\end{equation*}
for $\nu$ defined at the beginning of this proof.  Using \eqref{w1 w1'} this gives
\begin{align*}
B&\lesssim \frac{\e}{(1+\tau)^{\frac{1}{2}+\frac{\nu}{2}}} \sum_{|J|\leq |I|}  \l \frac{w^\half}{(1+|q|)^{1-\frac{\nu}{2}}} \tdr Z^{J}\psi^{(k'+1)}  \r_{L^2(B^R\cap\Sigma_\tau)}
\\&\lesssim  \frac{\e}{(1+\tau)^{\frac{1}{2}+\frac{\nu}{2}}} \sum_{|J|\leq |I|}  \l (w')^\half \tdr Z^{J}\psi^{(k'+1)}  \r_{L^2(B^R\cap\Sigma_\tau)}.
\end{align*}
Putting our estimates on $A$, $B$ and $C$ together we obtain from \eqref{energie psi k'+1}:
\begin{align*}
&E_{k'+1}(t)+S_{k'+1}(t) 
\\& \lesssim \sum_{|I|\leq N-2-5k'} \int_0^t (A+C)\l w^\half \dr Z^I\psi^{(k'+1)} \r_{L^2} \d\tau + \sum_{|I|\leq N-2-5k'} \int_0^t B\l w^\half \dr Z^I\psi^{(k'+1)} \r_{L^2} \d\tau
\\&\lesssim \e^3 \int_0^t \frac{\d\tau}{(1+\tau)^{\frac{3}{2}}} 
\\&\quad + \sum_{|J|\leq |I|\leq N-2-5k'} \int_0^t \frac{\e}{(1+\tau)^{\frac{1}{2}+\frac{\nu}{2}}}  \l (w')^\half \tdr Z^{J}\psi^{(k'+1)}  \r_{L^2(B^R\cap\Sigma_\tau)}\l w^\half \dr Z^I\psi^{(k'+1)} \r_{L^2} \d\tau
\\&\lesssim \e^3 \int_0^t \frac{\d\tau}{(1+\tau)^{\frac{3}{2}}} 
\\&\quad +\sum_{ |I|\leq N-2-5k'}\left(  \e \int_0^t   \l (w')^\half \tdr Z^{I}\psi^{(k'+1)}  \r_{L^2(B^R\cap\Sigma_\tau)}^2\d\tau + \e\int_0^t \frac{1}{(1+\tau)^{1+\nu}}\l w^\half \dr Z^I\psi^{(k'+1)} \r_{L^2} ^2\d\tau\right)
\\&\lesssim \e^3 \int_0^t \frac{\d\tau}{(1+\tau)^{\frac{3}{2}}} + \e S_{k'+1}(t) +  \e^3 \int_0^t \frac{\d\tau}{(1+\tau)^{1+\nu}}
\end{align*}
where for the integral involving $A+C$ we use \eqref{bootstrap sur ffi 2k+2} and for the integral involving $B$ we used first $2ab\leq a^2 + b^2$ and \eqref{bootstrap sur ffi 2k+2}. Therefore, we have proved that there exists a constant $C>0$ such that for all $t\in[0,T]$
\begin{equation*}
E_{k'+1}(t)+(1-C\e)S_{k'+1}(t)\leq C\e^3. 
\end{equation*}
We now choose $\e$ so that $\e\leq \frac{1}{2C}$. This implies that for all $t\in[0,T]$ 
\begin{equation*}
E_{k'+1}(t) \leq \frac{1}{2}\e^2.
\end{equation*}
Therefore, we improved \eqref{bootstrap sur ffi 2k+2}, thus proving that actually $T=T_0$. This breaks the blow-up criterion, proving that $T_0=+\infty$.  Note that since $\psi^{(k'+1)}$ is supported in $B_R$ the ghost weight $w$ is bounded and we can change \eqref{bootstrap sur ffi 2k+2} to the estimate of the lemma.
\end{proof}

This concludes the induction and thus the proof of Proposition \ref{prop récurrence}.

\subsection{The equation for the remainder}\label{section h}

The proof of the global existence for \eqref{equation sur h} follows the same structure as the proof of Proposition \ref{prop ffi}.  

\subsubsection{Bootstrap assumptions and first consequences}

We use again the ghost weight $w$ defined in Theorem \ref{theo principal}. We let $T_0$ be the maximal time of existence of a solution $h_{\lambda}$, which satisfies the following blow-up criterion:
\begin{equation*}
    T_0<+\infty \iff \lim_{t\to T_0^-}\sum_{|I|\leq N'}\l w^{\frac{1}{2}}\dr Z^Ih_{\lambda}\r_{L^2}(t)=+\infty,
\end{equation*}
where $N'$ is some large enough integer. In terms of number of derivation, the limiting term in the equation for $h_\lambda$ is $\Box F^{(K,i)}$, which we can differentiate at most $N+2-5K$ thanks to \eqref{B}. Therefore we set $N'=N+2-5K$, which satisfies $N'\geq 8$.  We define two rescaled energies:
\begin{align*}
\mathcal{E}_\lambda(t)&\vcentcolon=\sum_{|I|\leq N'}\lambda^{-2K+2|I|}\l w^{\frac{1}{2}}\dr Z^Ih_{\lambda} \r_{L^2}^2(t),
\\ \mathcal{S}_{\lambda}(t)&\vcentcolon=\sum_{|I|\leq N'}\lambda^{-2K+2|I|}\int_0^t\l (w')^{\frac{1}{2}}\tdr Z^Ih_{\lambda} \r_{L^2}^2(\tau)\d\tau.
\end{align*}
We make the following bootstrap assumptions on $h_{\lambda}$:
\begin{align}
    \sup_{t\in[0,T]}\mathcal{E}_\lambda(t)\leq C_2^2\e^2.\label{BA h 2}
\end{align}
Here $C_2$ is a constant to be choosen later. We let $T<T_0$ be the maximal time such that these estimates hold on $[0,T]$. Since $(h_\lambda,\dr_th_\lambda)_{|t=0}=0$, this bootstrap assumption is satisfied at $t=0$ and $T>0$. Note that \eqref{BA h 2} immediately implies that for every multi-index $|I|\leq N'$ and for every $t\in[0,T]$ we have
\begin{equation}
\l w^{\frac{1}{2}}\dr Z^I h_{\lambda} \r_{L^2}(t)\leq C_2\e\lambda^{K-|I|}.\label{BA h cool}
\end{equation}

Recall \eqref{equation sur h}:
\begin{align*}
\Box h_\lambda & =   \mathcal{N}(h_\lambda) + \lambda^K Q(\dr F,\dr\ffi)\T_{\lambda} + \lambda^K \Box F \T_\lambda
\\&\quad +\lambda^K \lambda^{+} \left( F\tdr F  \T_{\lambda} + Q(\dr F,\dr F) \T_{\lambda} + Q(\dr F,\dr\psi)\T_{\lambda} + F\tdr\psi\T_{\lambda} + Q(\dr\psi,\dr\psi) \right)
\end{align*}
To capture the different support properties, we define 
\begin{align*}
\mathbf{A}_\lambda & =  \lambda^K Q(\dr F,\dr\ffi)\T_{\lambda} + \lambda^K \Box F \T_\lambda +\lambda^K \lambda^{+} \left( F\tdr F  \T_{\lambda} + Q(\dr F,\dr F) \T_{\lambda} + Q(\dr F,\dr\psi)\T_{\lambda} + F\tdr\psi\T_{\lambda}  \right)
\\ \mathbf{B}_\lambda & =   \lambda^K \lambda^{+}Q(\dr\psi,\dr\psi)
\end{align*}
such that 
\begin{align}\label{decomposition}
\Box h_\lambda & =   \mathcal{N}(h_\lambda) + \mathbf{A}_\lambda + \mathbf{B}_\lambda.
\end{align}
Thanks to the vanishing of the initial data of $h_\lambda$ and the fact that the RHS of \eqref{decomposition} is supported in $B_R$, we can apply Lemma \ref{speed} to prove that $h_\lambda$ is also supported in $B_R$.

\subsubsection{Decay estimates}

We start by deriving some decay for $h_\lambda$ and the non-homogeneous terms in $\Box h_\lambda$. This again follows from the weighted Klainerman-Sobolev inequality and we don't repeat the argument. However, we state it in a lemma to emphasize the appearance of the factor $\lambda^{K-3-|I|}$.

\begin{lem}\label{rotation 2}
The following estimate holds on $[0,T]$
\begin{equation}
|  Z^I h_{\lambda}|\lesssim \frac{C_2\e\sqrt{1+|q|}}{1+s}\lambda^{K-3-|I|}\qquad\text{for}\quad|I|\leq N'-3.\label{WKNS 2}
\end{equation}
\end{lem}

The following lemma estimates the non-homogeneous terms in $\Box h_\lambda$.

\begin{lem}\label{rotation 1}
The function $\mathbf{A}_\lambda$ are supported in $A^R$ and satisfies
\begin{equation}
| Z^I\mathbf{A}_\lambda | \lesssim \frac{\e\lambda^{K-|I|}}{(1+s)^3}\label{estim Aq},\qquad\text{for}\quad|I|\leq N'.
\end{equation}
The function $\mathbf{B}_\lambda$ is supported in $B^R$ and satisfies 
\begin{equation}\label{Bq estimee}
|Z^I\mathbf{B}_\lambda|\lesssim \frac{\e^2}{(1+s)^3} \lambda^{K+1} \qquad\text{for}\quad|I|\leq N'.
\end{equation}
\end{lem}

\begin{proof}Since
\begin{equation}
 | Z^I\mathbb{T}_{\lambda}^{(\alpha,\beta)} |\leq C_{I,\alpha,\beta}\lambda^{-|I|},\label{estimee Z T}
\end{equation}
for some $C_{I,\alpha,\beta}>0$, we can directly estimate $\mathbf{A}_\lambda$ using its expression:
\begin{align*}
|Z^I\mathbf{A}_\lambda|&\lesssim \sum_{|I_1|+|I_2|\leq |I|}\lambda^{K-|I_2|}|Z^{I_1}\Box F| 
\\&\quad +\sum_{|I_1|+|I_2|+|I_3|\leq |I|}\lambda^{K-|I_3|}\frac{|Z^{I_1+1}F|}{1+r}\left( |Z^{I_2+1}F|+|Z^{I_2+1}\psi| + |Z^{I_2+1}\ffi| \right)
\\&\lesssim \frac{\e}{(1+s)^3}\sum_{|J|\leq|I|}\lambda^{K-|J|}
\end{align*}
where we used \eqref{A}, \eqref{B},\eqref{decay ffi l} and \eqref{estimee Z T}. For $\mathbf{B}_\lambda$ we recall \eqref{decay ffi l}, which implies that for all $2\leq \ell \leq K-1$
\begin{equation*}
|Z^I\psi^{(\ell)}|\lesssim \e\frac{\sqrt{1+|q|}}{1+s},\qquad\text{for}\quad |I|\leq M_{K-1}-3.\\
\end{equation*}
Recalling the expression of $\mathbf{B}_\lambda$ concludes the proof:
\begin{equation*}
|Z^I\mathbf{B}_\lambda|\lesssim\lambda^{K+1} \sum_{|I_1|+|I_2|\leq |I|}  \frac{|Z^{I_1+1}\psi||Z^{I_2+1}\psi|}{(1+s)(1+|q|)}\lesssim \frac{\e^2}{(1+s)^3}\lambda^{K+1}
\end{equation*}
since $N'+1\leq M_{K-1}-3$.
\end{proof}

\subsubsection{Energy estimates} We conclude the bootstrap by improving \eqref{BA h 2}.

\begin{prop}\label{last proposition}
For large enough $C_2$ and small enough $\e$ (depending on $C_2$), we have for every $t\in[0,T]$
\begin{equation*}
\mathcal{E}_\lambda(t)\leq\frac{1}{2}C_2^2\e^2.
\end{equation*}
\end{prop}

\begin{proof}
Applying Lemma \ref{weighted energy estimate} to $Z^I h_\lambda$ for every multi-index $|I|\leq N'$, multiplying what we obtain by $\lambda^{-2K+2|I|}$ and summing over $I$ we get for every $t\in[0,T]$ (recall that the initial data for $h_\lambda$ and $\dr_t h_\lambda$ vanish):
\begin{align}
\mathcal{E}_\lambda(t)+\mathcal{S}_\lambda(t)&\lesssim \sum_{|I|\leq N'}\int_0^t\lambda^{-2K+2|I|}\l w\Box Z^I h_\lambda \dr Z^I h_{\lambda}\r_{L^1}\d\tau\nonumber
\\&\lesssim  \sum_{|I|\leq N'}\int_0^t\lambda^{-2K+2|I|}\l w^\half\Box Z^I h_\lambda\r_{L^2}\l w^\half \dr Z^I h_{\lambda}\r_{L^2}\d\tau\label{energie pour h}
\end{align}
where we used the Cauchy-Schwarz inequality. Thanks to \eqref{decomposition} and \eqref{wave operator} we get for $|I|\leq N'$:
\begin{align*}
\l w^\half\Box Z^I h_\lambda\r_{L^2} & \lesssim \sum_{|J|\leq |I|}\left( \l w^\half Z^J\mathcal{N}(h_\lambda)  \r_{L^2(B_R\cap \Sigma_\tau)} + \l Z^J\mathbf{A}_\lambda  \r_{L^2(A_R\cap \Sigma_\tau)} +  \l Z^J\mathbf{B}_\lambda  \r_{L^2(B_R\cap \Sigma_\tau)}  \right).
\end{align*}
Thanks to Lemma \ref{rotation 1} we have
\begin{align}
 \l Z^J\mathbf{A}_\lambda  \r_{L^2(A_R\cap \Sigma_\tau)} +  \l Z^J\mathbf{B}_\lambda  \r_{L^2(B_R\cap \Sigma_\tau)} & \lesssim \e\lambda^{K-|J|} \left( \int_0^{\tau+R}\frac{\d r}{(1+r+\tau)^4} \right)^\half \lesssim \frac{ \e\lambda^{K-|J|}}{(1+\tau)^{\frac{3}{2}}}.\label{A et B}
\end{align}
We now look at $\mathcal{N}(h_\lambda)$ using \eqref{N(h)}:
\begin{align*}
\l w^\half Z^J\mathcal{N}(h_\lambda)  \r_{L^2} & \lesssim  \sum_{\substack{|J_1|+|J_2|\leq |J|\\ \Upsilon\in\{ \ffi,\psi,  \lambda^+F\T_\lambda \} \\j=1,2}} \left( N_1^{J_1,J_2} + N^{J_1,J_2}_{\Upsilon,j} \right)
\end{align*}
with
\begin{align*}
N_1^{J_1,J_2} & = \l w^\half\dr Z^{J_1}h_{\lambda}\tdr Z^{J_2}h_{\lambda} \r_{L^2},
\\ N_{\Upsilon,1}^{J_1,J_2} & = \l w^\half\dr Z^{J_1}h_{\lambda}\tdr Z^{J_2}\Upsilon \r_{L^2},
\\N_{\Upsilon,2}^{J_1,J_2} & = \l w^\half\dr Z^{J_1}\Upsilon \tdr Z^{J_2}h_{\lambda} \r_{L^2}.
\end{align*}

We start by $N_1^{J_1,J_2}$ and we distinguish two cases: since $|J_1| + |J_2| \leq N'$ we have either $|J_1|\leq \frac{N'}{2}$ or $|J_2|\leq \frac{N'}{2}$.  
\begin{itemize}
\item If $|J_1|\leq \frac{N'}{2}$, then \eqref{WKNS 2} implies (since $N'\geq 8 \Longrightarrow \frac{N'}{2}+1\leq N'-3$)
\begin{align*}
| \dr Z^{J_1} h_\lambda | \lesssim  \frac{ \e \lambda^{K-4-|J_1|}}{(1+s)^{\half + \frac{\nu}{2} } (1+|q|)^{ 1-\frac{\nu}{2}  }  }
\end{align*}
where $\nu$ is as in \eqref{w1 w1'}. Using \eqref{w1 w1'} we obtain
\begin{align}
N_1^{J_1,J_2} & \lesssim \frac{ \e \lambda^{K-4-|J_1|}}{(1+\tau)^{\half + \frac{\nu}{2} }   }\l  (w')^\half \tdr Z^{J_2} h_\lambda   \r_{L^2}\label{first case}.
\end{align}
\item If $|J_2|\leq \frac{N'}{2}$,  then \eqref{WKNS 2} implies
\begin{align*}
\tdr Z^{J_2}h_{\lambda} \lesssim \frac{\e  \lambda^{K-4-|J_2|}  }{(1+s)^{\frac{3}{2}}}.
\end{align*}
Therefore using \eqref{BA h cool} we obtain
\begin{align}
N_1^{J_1,J_2} & \lesssim \frac{\e  \lambda^{K-4-|J_2|}  }{(1+\tau)^{\frac{3}{2}}} \l w^\half\dr Z^{J_1}h_{\lambda} \r_{L^2} \lesssim \frac{\e^2  \lambda^{2K-4 - |J_1| -|J_2|}  }{(1+\tau)^{\frac{3}{2}}}\label{second case}.
\end{align}
\end{itemize}
We now turn to $N_{\Upsilon,1}^{J_1,J_2} $ for $\Upsilon \in \{ \ffi,\psi,  \lambda^+F\T_\lambda \}$. Using \eqref{decay ffi l}, \eqref{decay ffi} and \eqref{A} we obtain
\begin{align*}
|\tdr Z^{J_2} \Upsilon | \lesssim \frac{\e \lambda^{-|J_2|}}{(1+s)^{\frac{3}{2}}} 
\end{align*}
where we used the presence of $\lambda^+$ to regain one power of $\lambda$. Using then \eqref{BA h cool} this implies
\begin{align}
N_{\Upsilon,1}^{J_1,J_2} \lesssim \frac{\e \lambda^{-|J_2|}}{(1+\tau)^{\frac{3}{2}}} \l w^\half\dr Z^{J_1}h_{\lambda} \r_{L^2} \lesssim  \frac{\e^2 \lambda^{K-|J_1|-|J_2|}}{(1+\tau)^{\frac{3}{2}}}. \label{N i 1}
\end{align}
Finally, let us look at $N_{\Upsilon,2}^{J_1,J_2} $ for $\Upsilon \in \{ \ffi,\psi,  \lambda^+F\T_\lambda \}$. Using again \eqref{decay ffi l}, \eqref{decay ffi} and \eqref{A} we obtain
\begin{align*}
|\dr Z^{J_1} \Upsilon | \lesssim \frac{\e \lambda^{-|J_1|}}{(1+s)^{\half + \frac{\nu}{2} } (1+|q|)^{ 1-\frac{\nu}{2}  }  }
\end{align*}
where $\nu$ is as in \eqref{w1 w1'}. Therefore using \eqref{w1 w1'} we obtain
\begin{align}
N_{\Upsilon,2}^{J_1,J_2} \lesssim \frac{\e \lambda^{-|J_1|}}{(1+\tau)^{\half + \frac{\nu}{2} }  } \l  (w')^\half \tdr Z^{J_2} h_\lambda   \r_{L^2}.\label{N i 2}
\end{align}

\par\leavevmode\par
We now put \eqref{A et B}-\eqref{first case}-\eqref{second case}-\eqref{N i 1}-\eqref{N i 2} together into \eqref{energie pour h} to obtain
\begin{align*}
\mathcal{E}_\lambda(t)+\mathcal{S}_\lambda(t)&\lesssim \sum_{|J|\leq |I|\leq N'} \int_0^t  \frac{\e\lambda^{-K + 2|I| - |J|}}{(1+\tau)^{\frac{3}{2}}}\l w^\half \dr Z^I h_{\lambda}\r_{L^2}\d\tau 
\\&\quad + \sum_{|J_1|+|J_2| \leq |I| \leq N'} \int_0^t  \frac{\e^2\lambda^{-K + 2|I| - |J_1| - |J_2|}}{ (1+\tau)^{\frac{3}{2}} } \l w^\half \dr Z^I h_{\lambda}\r_{L^2}\d\tau
\\&\quad +  \sum_{|J_1|+|J_2|\leq |I|\leq N'}\int_0^t\frac{\e\lambda^{-2K+2|I|-|J_1|}}{(1+\tau)^{\half + \frac{\nu}{2} }  } \l  (w')^\half \tdr Z^{J_2} h_\lambda   \r_{L^2}\l w^\half \dr Z^I h_{\lambda}\r_{L^2}\d\tau
\end{align*}
where the first line corresponds to \eqref{A et B}, the second to \eqref{second case} and \eqref{N i 1} and the third to \eqref{first case} and \eqref{N i 2}. Note that we used $K\geq 4$ to bound $\lambda^{K-4}$ by 1 in \eqref{first case} and \eqref{second case}.  For the first two integrals, we simply use \eqref{BA h cool}. For the third integral we use $2ab\leq a^2+b^2$ and split the $\lambda$ powers carefully:
\begin{align*}
\int_0^t&\frac{\e\lambda^{-2K+2|I|-|J_1|}}{(1+\tau)^{\half + \frac{\nu}{2} }  } \l  (w')^\half \tdr Z^{J_2} h_\lambda   \r_{L^2}  \l w^\half \dr Z^I h_{\lambda}\r_{L^2} 
\\&\lesssim\int_0^t \e\lambda^{-2K + 2|J_2| }  \l  (w')^\half \tdr Z^{J_2} h_\lambda   \r_{L^2}^2 + \int_0^t \frac{\e\lambda^{  -2K + 4|I| - 2|J_1|-2|J_2|  }  }{(1+\tau)^{1+\nu } }  \l w^\half \dr Z^I h_{\lambda}\r_{L^2}^2
\\&\lesssim  \e \mathcal{S}_\lambda(t)  +  \e^3 \lambda^{2(|I| - |J_1|-|J_2|)}
\end{align*}
Since $|J_1|+|J_2|\leq |I|$ and $|J|\leq |I|$ we have proved that there exists a constant $C>0$ such that
\begin{align*}
\mathcal{E}_\lambda(t)+(1-C\e)\mathcal{S}_\lambda(t)\leq  CC_2\e^2+C\e^3
\end{align*}
We now choose $C_2\geq 4C$ and $\e\leq \frac{C_2^2}{4C}$ so that for every $t\in[0,T]$, $\mathcal{E}_\lambda(t)\leq \frac{C_2^2}{2}\e^2$.

\end{proof}

The result of Proposition \ref{last proposition} contradicts the maximality of $T$ and thus $T=T_0$. But now $\mathcal{E}_\lambda(t)$ is bounded up to $T_0$ by $C_2^2\e^2$, which proves that $T_0=+\infty$. This concludes the proof of Theorem \ref{theo principal}.

\appendix

\section{Commutators}\label{appendice B}

In this first appendix, we estimate $[Z,\tdr]$ and $[Z,\Ll]$. The later will be used in the second appendix, where we solve equation \eqref{Ll1}.

\begin{lem}\label{commutator tangent}
For every $Z\in\mathcal{Z}$, we have
\begin{equation*}
    [Z,\tdr_i]=p^{\ell}_i\tdr_{\ell}+\sum_{Z'\in\mathcal{Z}}\frac{p_{i,Z'}}{r}Z',
\end{equation*}
for some functions $p^{\ell}_i$ and $p_{i,Z'}$ such that for every $Z'\in\mathcal{Z}$, multi-index $K$ and $\ell\in\{1,2,3\}$ there exists $C^{K,\ell}_Z$ such that
\begin{equation*}
    |Z^Kp^{\ell}_i|+|Z^Kp_{i,Z'}|\leq C^{K,I,\ell}_J
\end{equation*}
in the region $A^R$.
\end{lem}

\begin{proof}
It follows from direct computations, using the formula $\tdr_i=\dr_i-\frac{x_i}{r}\dr_r$:
\begin{align*}
    [\tdr_i,\dr_t]&=0, \\
    [\tdr_i,\dr_j]&= \frac{\omega_i}{r}\tdr_j+\frac{1}{r}(\delta_{ij}-\omega_i\omega_j)\dr_r,\\
    [\tdr_i,S]&=\tdr_i,\\
    [\tdr_i,\Omega_{jk}]&=\delta_{ik}\tdr_j-\delta_{ij}\tdr_k,\\
    [\tdr_i,\Omega_{0j}]&=\frac{t\omega_i}{r}\tdr_j+\frac{\omega_k}{r}(\delta_{ij}-\omega_i\omega_j)\Omega_{0k}.
\end{align*}
\end{proof}

The following proposition estimates the commutator between the Minkowski vector fields and the transport operator $\Ll$.

\begin{prop}\label{commutator}
For every multi-index $I$ we have 
\begin{equation*}
    [\Ll,Z^I]=\sum_{|J|\leq|I|-1}\left( \frac{a^{I,\ell}_J}{r}\tdr_{\ell}Z^J+b^I_JZ^J\Ll+\frac{c^I_J}{r^2}Z^J \right),
\end{equation*}
for some functions $a^{I,\ell}_J$, $b^I_J$ and $c^I_J$ such that for every multi-index $I$, $J$ and $K$ with $|J|+1\leq |I|$ and every $\ell\in\{1,2,3\}$ there exists $C^{K,I,\ell}_J>0$ such that 
\begin{equation*}
    |Z^Ka^{I,\ell}_J|+|Z^Kb^I_J|+|Z^Kc^I_J|\leq C^{K,I,\ell}_J
\end{equation*}
in the region $A^R$.
\end{prop}

\begin{proof}
We prove this by strong induction on the value of $|I|$. Some direct computations show that
\begin{align*}
    [\Ll,\dr_t]=[\Ll,\Omega_{ij}]&=0,\\
    [\Ll,\dr_i]&=-\frac{1}{r}\tdr_i+\frac{\omega_i}{r^2}\mathrm{Id},\\
    [\Ll,\Omega_{0i}]&=\frac{t-r}{r}\tdr_i-\omega_i\Ll+\frac{\omega_i(r-t)}{r^2}\mathrm{Id},\\
    [\Ll,S]&=\Ll,
\end{align*}
which proves the proposition if $|I|=1$. We use the notation $[\Ll,Z]=\frac{a^{Z,\ell}_0}{r}\tdr_{\ell}+b^Z_0\Ll+\frac{c^Z_0}{r^2}\mathrm{Id}$. Now let $|I|$ be a multi-index with $|I|\leq N-1$ and $Z$ be any minkowskian vector field, we have 
\begin{align*}
    [\Ll,ZZ^I]&= Z[\Ll,Z^I]+[\Ll,Z]Z^I \\
    &= \sum_{|J|\leq|I|-1}\left( Z\left(\frac{a^{I,\ell}_J}{r}\right)\tdr_{\ell}Z^J+Z\left(b^I_J\right)Z^J\Ll+Z\left(\frac{c^I_J}{r^2}\right)Z^J \right)+\sum_{|J|\leq|I|}\left( b^I_JZ^J\Ll+\frac{c^I_J}{r^2}Z^J \right)\\& \qquad+\sum_{|J|\leq|I|-1}\frac{a^{I,\ell}_J}{r}Z\tdr_{\ell}Z^J+\frac{a^{Z,\ell}_0}{r}\tdr_{\ell}Z^I+b^Z_0\Ll Z^I+\frac{c^Z_0}{r^2}Z^I.
\end{align*}
Using the assumptions made on the functions $a^{I,\ell}_J$, $b^I_J$ and $c^I_J$, the fact that $|Z^K(r^{-1})|\lesssim r^{-1}$ for every multi-index $K$, we see that the first sum contains only good terms, meaning that they have the form we want. It is also the case for the terms in the second sum and for the terms $\frac{a^{Z,\ell}_0}{r}\tdr_{\ell}Z^I$ and $\frac{c^Z_0}{r^2}Z^I$. To conclude the proof, it remains to deal with the third sum and $b^Z_0\Ll Z^I$:
\begin{itemize}
    \item for $b^Z_0\Ll Z^I$, we simply write $b^Z_0\Ll Z^I=b^Z_0Z^I\Ll+b^Z_0[\Ll,Z^I]$. Because of the assumptions made on $b^Z_0$, it proves that the term $b^Z_0\Ll Z^I$ has the desired form.
    \item for the third sum, we use Lemma \ref{commutator tangent}:
    \begin{align*}
        \sum_{|J|\leq|I|-1}\frac{a^{I,\ell}_J}{r}Z\tdr_{\ell}Z^J&= \sum_{|J|\leq|I|}\frac{a^{I,\ell}_J}{r}\tdr_{\ell}Z^J+ \sum_{|J|\leq|I|-1}\frac{a^{I,\ell}_J}{r}[Z,\tdr_{\ell}]Z^J \\
        &=\sum_{|J|\leq|I|}\frac{a^{I,\ell}_J}{r}\tdr_{\ell}Z^J+ \sum_{|J|\leq|I|-1}\frac{a^{I,\ell}_Jp_{\ell}^k}{r}\tdr_{k}Z^J
        +\sum_{|J|\leq|I|-1}\frac{a^{I,\ell}_Jp_{\ell,Z'}}{r^2}Z'Z^J,
    \end{align*}
    which is of the desired form.
\end{itemize}
This concludes the proof of the proposition.
\end{proof}

The following Lemma estimates $\Box f$ if we know $\Ll f$.

\begin{lem}\label{d'alembertien}
If $f$ and $g$ are such that $\Ll f=g$, then
\begin{equation*}
    \Box f=\delta^{ij}\tdr_i\tdr_jf-\underline{L}g-\frac{g}{r}.
\end{equation*}
\end{lem}

\begin{proof}
Applying separately $-\dr_t + \dr_r$ to $\Ll f=g$ we obtain
\begin{equation*}
    -\dr_t^2f+\dr_r^2f-\frac{1}{r}\left( \dr_tf-\dr_rf+\frac{f}{r}\right)=-\underline{L}g.
\end{equation*}
We recognize on the LHS of the previous equation the operator $\Ll$ applied to $f$ so that :
\begin{equation*}
    -\dr_t^2f+\dr_r^2f+\frac{2}{r}\dr_rf=-\underline{L}g-\frac{g}{r}.
\end{equation*}
Recalling that $\Delta=\dr_r^2+\frac{2}{r}\dr_r+\delta^{ij}\tdr_i\tdr_j$ concludes the proof.
\end{proof}

\section{The transport equations}\label{appendice A}

The goal of this appendix is to prove Proposition \ref{equation de transport prop}.

\subsection{The reduced transport equation}

We start by considering the following reduced transport equation 
\begin{equation}\label{transport reduit 2}
(L+\eta)f=g
\end{equation}
where
\begin{itemize}
\item $g:\R_+\times\R^3\longmapsto\R$ is a continuous function supported in $A^R$,
\item $\eta:\R_+\times\R^3\longmapsto\R$ is a continuous function such that $|\eta|\lesssim\frac{\e}{(1+s)^2}$.
\end{itemize}
We start by constructing a well-chosen solution to $(L-\eta)\phi=0$ which is close to 1.

\begin{mydef}
Let $\chi:\R^3\to[0,1]$ be a smooth function supported in $\{\frac{1}{2R}\leq r \leq 2R\}$ and such that $\chi_{|A_0}\equiv 1$. Using polar coordinates we define $\hat{\chi}$ a cut-off function around $A^R$ in $\R_+\times\R^3$ by
\begin{equation*}
    \hat{\chi}(t,r,\omega)=\chi(r-t,\omega).
\end{equation*}
\end{mydef}

\begin{lem}
There exists a function $\beta$ vanishing outside of $\{\frac{1}{2R}\leq q \leq 2R\}$ which satisfies
\begin{equation*}
|\beta|\lesssim\e\qquad\text{and}\qquad L(1+\beta)=(1+\beta)\eta\quad\text{in $A^R$}.
\end{equation*}
\end{lem}

\begin{proof}
We simply define $\beta$ using polar coordinates
\begin{equation*}
\beta(t,r,\omega)\vcentcolon= \exp\left(\hat{\chi}(t,r,\omega)\int_0^t\eta(t',r-t+t',\omega)\d t' \right)-1.
\end{equation*}
Thanks to our assumption on $\eta$, $\beta$ satisfies the bound $|\beta|\lesssim\e$. Thus, if $\e$ is small enough, we can bound $1+\beta$ from below by $\frac{1}{2}$, and therefore $\ln(1+\beta)$ is well-defined and satisfies clearly $L(\ln(1+\beta))=\eta$ in $A^R$, which implies $L(1+\beta)=(1+\beta)\eta$.
\end{proof}

It is now easy to solve \eqref{transport reduit 2}:

\begin{lem}\label{transport reduit lemme}
Let $f_0:\R^3\to\R$ supported in $A_0$ and $g:\R_+\times\R^3\to\R$ supported in $A^R$. There exists a unique global solution $f$ of \eqref{transport reduit 2} with initial data $f_0$. This solution is supported in $A^R$ and satisfies 
\begin{equation*}
    |f|\lesssim \sup_{(t,x)\in\R_+\times\R^3}\left( |f_0(x)|+(1+|x|)^2|g(t,x)|\right).
\end{equation*}
\end{lem}

\begin{proof}
The uniqueness follows easily from an energy estimate related to the operator $L+\eta$. To construct a solution, we use polar coordinates and set
\begin{align*}
    f(t,r,\omega)\vcentcolon & = \frac{1}{1+\beta(t,r,\omega)}  \Bigg( (1+\beta(0,r-t,\omega))f_0(r-t,\omega) 
    \\&\qquad\qquad\qquad\qquad\qquad  +\int_0^t(1+\beta(s,r-t+s,\omega))g(s,r-t+s,\omega)\d s \Bigg).
\end{align*}
 By construction, we have $f_{|t=0}=f_0$ and $L((1+\beta)f)=(1+\beta)g$, and $f$ is supported in $A^R$. Let $(t,x)\in A^R$, since $L(1+\beta)=(1+\beta)\eta$ in $A^R$ we have:
\begin{equation*}
    (L+\eta)f(t,x)=\frac{1}{1+\beta(t,x)}L((1+\beta)f)(t,x)=g(t,x).
\end{equation*}
This proves the existence of a solution. The estimate on $|f|$ comes from the expression of $f$, the estimate $\frac{1}{2}\leq 1+\beta\leq 2$ and the support property of $g$ (note that in $A^R$, $r$ and $t$ are equivalent).
\end{proof}

\subsection{The transport operator $\Ll$} Using the result of the previous section, we now solve the equation
\begin{equation}\label{Ll}
\Ll f = f \mu +g
\end{equation}
where 
\begin{equation}\label{assumption}
(1+r)^2|Z^I\mu|+(1+r)^3|Z^Ig|\lesssim \e 
\end{equation}
for $|I|\leq M$, where $M$ is an integer. We consider initial data $f_0$ supported in $A_0$ and satisfying $\l f_0\r_{C^M}\lesssim \e$. Our strategy is to first deduce some \textit{a priori} estimates.
\par\leavevmode\par
We first take advantage of $[\Ll,\Omega]=0$:
\begin{lem}\label{induction 1 lemme}
If $f$ is a smooth and global in time solution of \eqref{Ll} initially supported in $A_0$, then
\begin{equation}\label{induction 1}
|\Omega^J f|\lesssim \frac{\e}{1+r} \qquad \text{for}\quad |J|\leq M.
\end{equation}
\end{lem}

\begin{proof}
We prove this lemma by strong induction on the value of $|J|$. If $|J|=0$, we just rewrite \eqref{Ll} as
\begin{equation*}
(L-\mu)(rf)=rg
\end{equation*}
and apply Lemma \ref{transport reduit lemme} using \eqref{assumption}. Assume now that the estimate \eqref{induction 1} holds for all value of $|J|\leq k$ with $k\leq M-1$ and let $J$ be a multi-index such that $|J|=k+1$. Since $[\Ll,\Omega]=0$, the equation for $\Omega^J f$ is 
\begin{equation*}
\Ll \Omega^J f = \Omega^J f \mu + \sum_{\substack{J_1+J_2=J\\|J_1|\leq k}}\Omega^{J_1}f \Omega^{J_2}\mu + \Omega^J g
\end{equation*}
which rewrites
\begin{equation*}
(L-\mu)(r\Omega^J f)= \sum_{\substack{J_1+J_2=J\\|J_1|\leq k}}r\Omega^{J_1}f \Omega^{J_2}\mu + r\Omega^J g.
\end{equation*}
Because of the induction hypothesis and of \eqref{assumption}, the RHS of this equation is bounded by $\frac{\e}{(1+r)^2}$. Thus, we can apply Lemma \ref{transport reduit lemme} and conclude the proof.
\end{proof}

Note that in the previous lemma, in order to apply Lemma \ref{transport reduit lemme}, we have to prove that the initial data for $r\Omega^Jf$ are bounded. Since the vector fields $\Omega$ involve only spatial derivations, this boundedness property follows from $f$ being supported in $A_0$ and smooth. In the next lemma however, we fully use Proposition \ref{commutator} and estimate $\Omega^J Z^I f$, again with the help of Lemma \ref{transport reduit lemme}. This requires to bound the initial data for $r\Omega^J Z^I f$, which may involve some time derivative. To find $\dr^k_tf_{|t=0}$, we simply use the equation \eqref{Ll} and $[\Ll,\dr_t]=0$ to derive the equation for $\dr^{k-1}_tf$ which gives 
\begin{equation*}
\dr^k_t f=-\dr\dr_t^{k-1} f -\frac{\dr_t^{k-1} f}{r}+\sum_{i+j=k-1}\dr^i_tf\dr^j_t\mu+\dr^{k-1}_tg.
\end{equation*}
Therefore we obtain $\dr^k_tf_{|t=0}$ by induction on $k$, and check easily their initial boundedness. With that in mind, we can estimate $\Omega^J  Z^I f$:

\begin{lem}\label{induction 2 lemme}
If $f$ is a smooth and global in time solution of \eqref{Ll} initially supported in $A_0$, then
\begin{equation}\label{induction 2}
|\Omega^J  Z^I f|\lesssim \frac{\e}{1+r} \qquad \text{for}\quad |I|+|J|\leq M.
\end{equation}
\end{lem}

\begin{proof}
We prove this lemma by strong induction on the value of $|I|$. The case $|I|=0$ is a consequence of Lemma \ref{induction 1 lemme}. Assume now that the estimate \eqref{induction 2} holds for all value of $|I'|\leq k$ with $k\leq M-1$ and all $J'$ such that $|I'|+|J'|\leq M$. Let $I$ be a multi-index such that $|I|=k+1$ and let $J$ be a multi-index such that $|I|+|J|\leq M$. Formally, the equation on $\Omega^J  Z^I f$ is 
\begin{equation}\label{induction step}
(L-\mu)(r\Omega^JZ^If) = \sum_{\substack{I_1+I_2=I\\J_1+J_2=J\\|I_1|\leq k\\|J_1|\leq |J|-1 }}r\Omega^{J_1}Z^{I_1} f \Omega^{J_2} Z^{I_2}\mu +r\Omega^JZ^Ig + r[\Ll,\Omega^JZ^I]f .
\end{equation}
Using the induction hypothesis and \eqref{assumption}, the first two terms in the RHS of \eqref{induction step} are bounded by $\frac{\e}{(1+r)^2}$. It remains to estimate the commutator. Using $[\Ll,\Omega]=0$ and Proposition \ref{commutator} we have:
\begin{align}\
|r[\Ll,\Omega^JZ^I]f| & \lesssim \left|r\Omega^J\sum_{|I'|\leq|I|-1}\left( \frac{a^{I,\ell}_{I'}}{r}\tdr_{\ell}Z^{I'}f+b^I_{I'}Z^{I'}\Ll f+\frac{c^I_{I'}}{r^2}Z^{I'}f\right)\right|\nonumber
\\& \lesssim \sum_{|I'|\leq|I|-1}  \left|\Omega^J \tdr Z^{I'} f\right|+\sum_{|I'|\leq|I|-1}  r\left|\Omega^J Z^{I'}\Ll f\right|+\sum_{|I'|\leq|I|-1}  \frac{1}{r}\left|\Omega^JZ^{I'}f\right| \label{blablabla}
\end{align}
Since $|I'|\leq |I|-1$ we can use the induction hypothesis to obtain
\begin{equation*}
\sum_{|I'|\leq|I|-1}\frac{1}{r}\left|\Omega^JZ^{I'}f\right|\lesssim \frac{\e}{(1+r)^2}.
\end{equation*}
To estimate the terms involving $\Ll f$, we simply use \eqref{Ll}, the induction hypothesis and \eqref{assumption}:
\begin{align*}
\sum_{|I'|\leq|I|-1}r\left|\Omega^J Z^{I'}\Ll f\right| & \lesssim  \sum_{\substack{|I'|\leq |I|-1\\I_1+I_2=I'\\J_1+J_2=J}}r|\Omega^{J_1}Z^{I_1} f \Omega^{J_2} Z^{I_2}\mu| +r|\Omega^JZ^Ig| \lesssim \frac{\e}{(1+r)^2}.
\end{align*}
It remains to estimate the first terms in \eqref{blablabla}. Using $[\Omega,\tdr]\sim\tdr$ and $\tdr_\ell=\frac{\omega^k}{r}\Omega_{\ell k}$ we get 
\begin{align*}
\sum_{|I'|\leq|I|-1}  \left|\Omega^J \tdr Z^{I'} f\right| \lesssim \sum_{\substack{|I'|\leq|I|-1\\ |J'|\leq |J|+1}}\frac{1}{r}\left| \Omega^{J'}Z^{I'}f \right|
\end{align*}
In this last sum, note that $|I'|\leq |I|-1$ and $|I'|+|J'|\leq |I|-1+|J|+1\leq M$. Thus we can use our induction hypothesis to bound this by $\frac{\e}{(1+r)^2}$. In conclusion, the RHS of \eqref{induction step} is bounded by $\frac{\e}{(1+r)^2}$, and we can apply Lemma \ref{transport reduit lemme} and conclude the proof.
\end{proof}

In the previous lemma, the case $J=0$ gives that a global smooth solution of \eqref{Ll} satisfies
\begin{equation*}
| Z^I f|\lesssim \frac{\e}{1+r} \qquad \text{for}\quad |I|\leq M.
\end{equation*}
This \textit{a priori} estimates allows us easily to prove the existence and uniqueness of a global solution to \eqref{Ll}. We summarize this discussion in the following corollary:

\begin{coro}
Let $f_0$ supported in $A_0$ and such that and let $\mu$ and $g$ satisfy \eqref{assumption}. There exists a unique global solution $f\in C^M$ to \eqref{Ll} with initial data $f_0$. Moreover, $f$ is supported in $A^R$ and satisfies
\begin{equation*}
| Z^I f|\lesssim \frac{\e}{1+r} \qquad \text{for}\quad |I|\leq M.
\end{equation*}
\end{coro}

\nocite{*}
\bibliographystyle{plain}

\end{document}